\documentclass[a4paper,10pt,leqno]{amsart}
\usepackage{amsfonts}
\usepackage{amsmath}
\usepackage{amssymb}
\usepackage[utf8]{inputenc}
\usepackage{hyperref}
\usepackage{color}
\usepackage[shortlabels]{enumitem}

\usepackage{graphicx}

\newcommand\C{\mathbb{C}}
\newcommand\R{\mathbb{R}}
\newcommand\Q{\mathbb{Q}}
\newcommand\Z{\mathbb{Z}}
\newcommand\N{\mathbb{N}}
\newcommand\HH{\mathbb{H}}

\newtheorem{theorem}{{Theorem}}[section]
\newtheorem{proposition}[theorem]{{Proposition}}%[section]
\newtheorem{isom.ext}[theorem]{{Trivial isometric extension}}%[section]
\newtheorem{definition}[theorem]{{Definition}}%[section]
\newtheorem{lemma}[theorem]{{Lemma}}%[section]
\newtheorem{corollary}[theorem]{{Corollary}}%[section]
\newtheorem{remark}[theorem]{{Remark}}%[section]
\newtheorem{example}[theorem]{{Example}}%[section]
\newenvironment{Ex}{\begin{example} \normalfont}{\end{example}}

\definecolor{purple}{rgb}{0.65,0.12,0.94}
\definecolor{forestgreen}{rgb}{0.4,0.64,0.13}
\hypersetup{
	colorlinks=true,
	linkcolor=red,
	filecolor=magenta,      
	urlcolor=black,
}

\title{Arithmetic structure of generalized Inoue--Bombieri manifolds}
\author [B. Flamencourt]{Brice Flamencourt}
\address{UMPA, CNRS, 
	\'Ecole Normale Sup\'erieure de Lyon, France}
\email{brice.flamencourt@ens-lyon.fr}

\author[A. Zeghib]{Abdelghani Zeghib}
\address{UMPA, CNRS, 
	\'Ecole Normale Sup\'erieure de Lyon, France}
\email{abdelghani.zeghib@ens-lyon.fr 
	\hfill\break\indent
	\url{http://www.umpa.ens-lyon.fr/~zeghib/}}


\begin{document}
\maketitle

\begin{abstract}
A Generalized Inoue--Bombieri (GIB) manifold $M$ is a compact quotient of a connected Riemannian product $\R^q \times (N,g _N)$ by a discrete subgroup of $\mathrm{Sim}(\R^q) \times \mathrm{Isom}(N,g_N)$. The flat factor induces a transversely Riemannian foliation whose leaf closures determine, up to a natural geometric modification, a torus fibration $M \to X$. The main goal of this article is to study the associated monodromy representation $\rho : \pi_1(X) \to \mathrm{GL}(n,\Z)$. We prove that the image of $\rho$ is a subgroup of a cocompact arithmetic lattice of a reductive group, and we discuss which groups may be realized as monodromy groups of GIB manifolds.

When $(N,g_N)$ is a symmetric space of non-compact type, the monodromy itself is arithmetic. Moreover, one may describe the fibration and the monodromy in terms of parabolic subgroups of the isometry group of $(N,g_N)$. This yields new examples of GIB manifolds, as well as obstructions, and opens the way toward a complete classification in this particular case.
\end{abstract}

\tableofcontents

\section{Introduction}

\subsection{GIB manifolds} Generalized Inoue--Bombieri (GIB) manifolds are compact manifolds admitting a natural Riemannian foliation. They were introduced in \cite{FZ2} with the aim of simplifying the framework for the study of locally conformally product manifolds (LCP).

Recall that a similarity (also called a homothety) between two Riemannian manifolds $(M_1,g_1)$ and $(M_2, g_2)$ is a diffeomorphism $\phi : M_1 \to M_2$ such that $\phi^*g_2 = \lambda^2 g_1$ for some real number $\lambda > 0$ called the ratio of the similarity. The group of similarities from a Riemannian manifold $(M,g)$ to itself is denoted by $\mathrm{Sim}(M)$. A {\em strict similarity} is a similarity whose ratio is different from $1$. This leads to the following definition of GIB manifolds:

\begin{definition} \label{GIBdef}
A generalized Inoue--Bombieri (GIB) manifold is the compact quotient of a product manifold $\R^q \times N$, where $q \ge 1$ and $(N, g_N)$ is a connected Riemannian manifold, by a discrete subgroup $\Gamma$ of $\mathrm{Sim} (\R^q) \times \mathrm{Isom}(N,g_N)$ acting properly, freely and cocompactly, and whose projection onto $\mathrm{Sim}(\R^q)$ contains a strict similarity.
\end{definition}

GIB manifolds are named after the Inoue--Bombieri surfaces of type $S^0$, which they generalize. These are complex manifolds obtained as compact quotients of $\C \times \HH^2$ by discrete subgroups of $\mathrm{Aut}(\C) \times \mathrm{Aut}(\HH^2) = \mathrm{Sim}^+(\R^2) \times \mathrm{Isom}(\HH^2)$.

Note that, throughout this text, $(N, g_N)$ will always denote the Riemannian factor of the universal cover of a GIB manifold on which the fundamental group acts by isometries.

\subsection{Monodromy group} \label{intromonodromy} The initial motivation for the present work was to understand the arithmetic data that arise naturally from the definition of GIB manifolds. To introduce them, we fix a GIB manifold $M$  of the form $\Gamma \backslash (\R^q \times N)$ and observe that the submersion $\R^q \times N \to N$ induces a foliation that descends to a Riemannian foliation $\mathcal F$ on the quotient. 
 
It was shown in \cite[Theorem~1.6]{FZ2} that, up to a natural geometric transformation, the closures of the leaves of $\mathcal F$ are the fibers of a torus fibration $M \to X$  over a compact base $X$ which inherits a natural Riemannian metric. This gives rise to a {\em monodromy representation} $\rho: \pi_1(X) \to \mathrm{GL}(n,\Z)$, where $n$ denotes the dimension of the tori (see \cite{Fla25, FZ2, Kou}, and  Section~\ref{SecMonodromy}  for  more details). Its image, usually denoted $D$ is called the {\em monodromy group} of the GIB manifold.

Using structural properties of GIB manifolds, we show that, as a subgroup of $\mathrm{GL}(n,\Z)$, the monodromy group $D$ satisfies the following several strong constraints:

 \begin{itemize} 
 \item there exists a nontrivial decomposition $\R^n = H \oplus V$ that is preserved by $D$. (In the terminology of GIB manifolds, 
 the subspace $H$ corresponds to the Euclidean factor $\R^q$ of the universal cover).
\item $H$ is an irrational subspace of $\R^n$, meaning that $H + \Z^n$ is dense in $\R^n$. 
\item $D$ acts on $H$ by similarities with respect to a fixed inner product $b$, and not all of these similarities are isometries.
\end{itemize}

\subsection{GIB arithmetic data} From the  previous properties of the monodromy group $D$ of a GIB manifold, we can define a purely  algebraic structure, that we will call {\em GIB arithmetic data}. It consists in giving a tuple $(n,H,V,b)$, where $\R^n = H \oplus V$, the subspace $H$ is irrational, and $b$ is an inner product on $H$. We  then define the automorphism group $\mathrm{Aut}(n,H,V,b)$  as  those elements of $\mathrm{GL}(n,\Z)$  preserving $H$ and $V$ and acting by similarities with respect to $b$ (on $H$).  We assume (in our definition of GIB arithmetic data) that this group  $\mathrm{Aut}(n,H,V,b)$ contains elements acting  by strict similarities on $H$.

Our main result about this abstact structure is: 

\begin{theorem} \label{structureThm} The automorphism group  $ {\bf D}= \mathrm{Aut}(n, H, V, b)$ of GIB arithmetic data  $(n, H, V, b)$ is a cocompact arithmetic group. More precisely  let $G$ be this Zariski closure. Then, $G$ is a $\Q$-anisotropic reductive group. In particular, $G$ is the almost-direct product of its center $T$ with a semisimple group $S$, and ${\bf D}_T :=  { \bf D} \cap T \neq 1$ and ${\bf D}_S :=  {\bf D} \cap S$ are cocompact lattices in $T$ and $S$ respectively. 
\end{theorem}

\begin{remark}
Actually, we prove that Zariski closures of general subgroups $D \subset \mathrm{Aut}(n, H, V, b)$ are reductive, and that the monodromy subgroups are arithmetic in some geometric contexts.
\end{remark}

A natural question is, which arithmetic groups of reductive groups arise in this way, i.e. are equal to some  $\mathrm{Aut}(n, H, V, b)$? In particular which $D_S$, the arithmetic group in the semisimple part one can realize?

We also ask which abstract subgroups $D \subset  \mathrm{Aut}(n, H, V, b)$, can be realized as monodromy groups of a GIB manifold. But since the nomodromy does not determine the geometry of $N$ (neither $X$, the base of the toral fibration $M \to X$), then, given a realized monodromy $D$, what best geometry can be associated to it?  The question can be in particular asked in the case of the full mondoromy group $ \mathrm{Aut}(n, H, V, b)$. Partial answers are in the sequel.

\subsection{Topology of GIB manifolds with a given monodromy group} As we wrote before, GIB manifolds are an alternative point of view for LCP manifolds. These latter manifolds were previously studied in \cite{Fla24} in the context of conformal geometry. Their precise definition emerged from a detailed analysis of Weyl structures on compact manifolds initiated by Belgun and Moroianu in \cite{BM}.

A Weyl structure is a torsion-free connection on a conformal manifold that preserves the conformal structure. Such a structure is called closed and non-exact if it is locally, but not globally, the Levi-Civita connection of a metric in the given conformal class. The original question was about the reduced holonomy of a Weyl connection on a compact manifold, which was conjectured to be flat or irreducible. However, a third situation was found by Matveev and Nikolayevsky \cite{MN15}, where the universal cover of the manifold, endowed with the lifted foliation, is a Riemannian product of an Euclidean space and an irreducible manifold. Kourganoff later proved in \cite{Kou} that this exhausts all possible cases. The authors of the present paper provided a new proof of this result in \cite{FZ1}, relying more heavily on geometric arguments.

LCP manifolds are quotients of a Riemannian product $\R^q \times (N,g_N)$, where $(N,g_N)$ is irreducible, by a discrete subgroup of $\Gamma \subset \mathrm{Sim}(\R^q \times (N, g_N)) \cap (\mathrm{Sim}(\R^q) \times \mathrm{Sim}(N, g_N))$, that does not consist solely of isometries. In \cite{FZ2}, we observed that, after a conformal change of the metric $g_N$, this setting is equivalent to the case where $\Gamma \subset \mathrm{Sim}(\R^q) \times \mathrm{Isom}(N, g_N)$, corresponding to GIB manifolds.

The classification of LCP manifolds, and subsequently of GIB manifolds, is the main focus of their study. We aim to understand them through the analysis of invariants arising from their structure. It is therefore natural to ask whether GIB arithmetic data, which are now purely algebraic objects, always arise from a GIB manifold. The answer is affirmative, as shown by the following theorem.

\begin{theorem} \label{GIBmonodromy}
Let $(n,H,V,b)$ be GIB arithmetic data. We call a finitely generated subgroup $D$ of $\mathrm{Aut}(n,H,V,b)$ that contains an element acting as a strict similarity on $H$ a {\em GIB monodromy subgroup} of $(n,H,V,b)$. Then, GIB monodromy subgroups of GIB arithmetic data are exactly the monodromy groups of GIB manifolds.
\end{theorem}

Theorem~\ref{GIBmonodromy} has some flaws, as the geometry of the GIB manifold constructed from GIB arithmetic data is not known. However, an explicit construction yields:

\begin{proposition} \label{ExplicitMon}
In Theorem~\ref{GIBmonodromy}, and in the notation of Section~\ref{intromonodromy}, the universal cover $\tilde X$ of $X$ can be taken to be the  product of a Euclidean space and $\HH^2$.

If the finitely generated subgroup of $D \subset \mathrm{Aut}(n,H,V,b)$ taken in Theorem~\ref{GIBmonodromy} is cocompact in its Zariski closure, as it is the case for $\mathrm{Aut}(n,H,V,b)$, $\tilde X$ can be taken to be the symmetric space associated to this Zariski closure.
\end{proposition}

\subsection{The case of symmetric spaces}
We recalled earlier that GIB manifolds were introduced as a generalization of Inoue–Bombieri surfaces. For these surfaces, and with the notation of Definition~\ref{GIBdef} (which we continue to use in what follows), the manifold $N$ is the symmetric space $\HH^2$. In fact, the case where $N$ is a symmetric space is the most natural generalization of these surfaces, and a complete understanding of this situation is one of the main goals of this work.

In \cite{FZ2}, we already focused on this case, proving a Bieberbach-type rigidity result. More precisely, there exists a connected group $L$ acting properly and transitively on $\R^q \times N$ and containing $\Gamma$. In particular, $\Gamma$ is a cocompact lattice in $L$. Moreover, we classified GIB manifolds for which $N$ is a symmetric space of rank~$1$, showing that in this case $N$ must necessarily be hyperbolic space.

The main tool in this latter classification was the projection of $\Gamma$ onto $\mathrm{Isom}(N,g_N)$. The closure of this group, denoted by $\bar P$, has an abelian identity component $\bar P^0$. The fact that this group is abelian imposes very strong restrictions on $N$. We are therefore led to ask what $\bar P^0$ can be in this setting.

This problem was studied in a different context by del Barco and Moroianu in \cite{dBM}. They considered LCP manifolds, which are equivalent to GIB manifolds, whose universal cover $G$ carries a Lie group structure and such that $\Gamma$ is a subgroup of $G$. In that situation, the factor $N$ is a subgroup of $G$, and they proved that $\bar P^0$ is contained in the radical of $N$. In particular, it is isomorphic to $\R^m$ for some $m>0$. By the Bieberbach rigidity result mentioned above, our setting with $N$ symmetric is close to the Lie group case: indeed, if the group $L$ acted simply transitively on $\R^q \times N$, the universal cover would itself inherit a Lie group structure.

When $N$ is a symmetric space of non-compact type, its isometry group is algebraic. Using the theory of linear algebraic groups, we prove that $\bar P^0$ is contained in the unipotent radical of a suitable group, and more precisely we obtain the following result.

\begin{theorem}\label{P0UnipotentR}
Let $\Gamma \backslash (\R^q \times N)$ be a GIB manifold, where $(N,g_N)$ is a Riemannian  symmetric space of non-compact type. Let $\bar P$ be the closure of the projection of $\Gamma$ onto $\mathrm{Isom}(N,g_N)$. Then, the identity component $\bar P^0$ of $\bar P$ is the unipotent radical of any algebraic subgroup of $\mathrm{Isom}(N, g_N)^0$ lying between the Lie groups $\bar P$ and the normalizer $\mathrm{Nor}_G(\bar P^0)$. Moreover, $\mathrm{Nor}_G(\bar P^0)$ is an $\R$-parabolic subgroup of $G$.

Moreover, the basis X of the toral fibration $M \to X$, is a locally symmetric manifold of non-positive curvature, with a non-trivial Euclidean (local) factor. Equivalently, the Riemannian quotient $N / {\bar P}^0$ is a symmetric space of non-compact type.
\end{theorem}

This result opens the way to a classification of admissible non-compact symmetric spaces by means of the classification of parabolic subgroups of semisimple groups. It also suggests methods for constructing new examples involving non-compact symmetric spaces of rank greater than~$2$. Combined with results on monodromy groups, it furthermore yields obstructions for certain symmetric spaces to appear as the factor $N$. See Section~\ref{CounterEx} below for  both    constructions as well as obstructions regarding the symmetric spaces associated to $\mathrm{SL}(m, \R)$ and $\mathrm{SO}(p, r)$.

We notice that Theorem~\ref{P0UnipotentR} is partially implied by the work of Moore in \cite{Moo}. In this paper, closed cocompact subgroups of semisimple Lie groups with abelian identity component (like the group $\bar P$ in our situation) are studied. This setting was suggested by the work of D. Epstein \cite{Eps} on transversely hyperbolic foliations. In this latter article, the foliation under consideration is $1$-dimensional, but the framework of GIB manifolds offers examples of transversely Riemannian foliations of any dimension where the closures of the leaves are flat tori. This explains the similarities between the result of Moore and ours.

The monodromy of a GIB manifold in the setting of Theorem~\ref{P0UnipotentR} has the following important property:

\begin{theorem} \label{SymZar}
In the setting of Theorem~\ref{P0UnipotentR}, the monodromy group is cocompact in its Zariski closure and is thus arithmetic. Moreover, if $(N,g_N)$ is irreducible, then the monodromy representation is injective.
\end{theorem}

\subsection{Organization of the paper}

In Section~\ref{NotPre}, we introduce notation and recall some basic facts about arithmetic groups. In particular, we discuss restriction of scalars, which plays an important role throughout the paper. Section~\ref{SecPre} is devoted to precise definitions of our objects of study, namely monodromy groups and GIB arithmetic data. In Section~\ref{SecZarclos}, we prove Theorem~\ref{structureThm} and investigate the constraints on the semisimple part of the Zariski closure of a monodromy group.

Theorem~\ref{GIBmonodromy} is proved in Section~\ref{FromtoGIB}, where we give an explicit description of GIB manifolds admitting a prescribed monodromy group. GIB manifolds of the form $\Gamma \backslash (\R^q \times N)$ with $N$ symmetric are studied in Section~\ref{SecSym}, where Theorems~\ref{P0UnipotentR} and \ref{SymZar} are established. Finally, we discuss several examples of symmetric spaces that give rise to GIB manifolds, and we classify GIB arithmetic data from the matrix point of view in Section~\ref{SecExamples}.

\section{Notations and preliminaries} \label{NotPre}

\subsection{Notations}
Let $G$ be a Lie group. We denote by $G^0$ the identity component of $G$. If $H$ is a subgroup of $G$, the normalizer of $H$ in $G$ is denoted by $\mathrm{Nor}_G(H)$. The center of $G$ is denoted by $C(G)$, and the centralizer of $H$ in $G$ is denoted by $C_G(H)$.

The radical of $G$ is the maximal connected normal solvable subgroup of $G$; it is denoted by $R(G)$. If $G$ is a matrix group, we denote by $R_u(G)$ its unipotent radical, that is, the maximal connected normal unipotent subgroup of $G$. It is contained in $R(G)$.

Recall that a linear algebraic group over a field $k$ is said to be reductive if its unipotent radical over an algebraic closure of $k$ is trivial.

Throughout this text, we consider two topologies: the usual topology and the Zariski topology. If $H$ is a subgroup of $G$, we denote by $\bar H$ the closure of $H$ in $G$ with respect to the usual topology. The Zariski closure of $H$ in $G$ is denoted by $\overline H^\mathrm{Zar}$.

When no confusion is possible, we use gothic script to denote Lie algebras. For instance, $\mathfrak g$ denotes the Lie algebra of $G$.

Let $G$ be a linear algebraic group defined over a field $k$. We denote by $G(k)$ the $k$-rational points of $G$, i.e. the elements of $G$ with entries in $k$.

\subsection{Arithmetic groups and lattices}
In this section, we recall some basic facts about lattices and arithmetic groups, since these notions will play an important role in our analysis. We mainly follow the exposition in \cite{Mor}.

Let $G$ be a linear Lie group, i.e. a subgroup of $\mathrm{GL}(V)$ for some vector space $V$ of dimension $n$ that we fix from now on. For any subring $R$ of $\R$, we denote by $G(R)$ the elements of $G$ with entries in $R$.

Informally, an arithmetic group is the set of integer points of a linear Lie group. However, we define arithmeticity up to isomorphism and compact factors.

\begin{definition} \label{ArithGroup}
A subgroup $\Gamma$ of $G$ is an arithmetic subgroup if there exists a linear Lie group $G'$, two compact normal subgroups $K$ and $K'$ of $G$ and $G'$ respectively, and an isomorphism
\[
\phi : G/K \to G'/K'
\]
such that $\phi(\pi(\Gamma))$ is commensurable with $\pi'(G'(\Z))$, where $\pi : G \to G/K$ and $\pi' : G' \to G'/K'$ are the canonical projections.
\end{definition}

Definition~\ref{ArithGroup} will not be useful in our study, and we  need a less abstract point of view. We start by defining a $k$-form of a finite-dimensional vector space $V$ for an algebraic number field $k$.

\begin{definition}
A $k$-form of $V$ is a $k$-vector subspace $V_k$ of $V$ such that the natural map $V_k \otimes_k \R \to V$ is an isomorphism.
\end{definition}

We fix a $k$-form $V_k$ of $V$. It induces a $k$-form $\mathrm{End}(V)_k$ of the space $\mathrm{End}(V)$ of endomorphisms of $V$ defined by
\begin{equation}
\mathrm{End}(V)_k := \{A \in \mathrm{End}(V) \ \vert \ A(V_k) \subset V_k\}.
\end{equation}

A function $P: V \to \R$ on $V$ is polynomial if $P \circ f$ is polynomial for any isomorphism $f: \R^n \to V$. Moreover, we say that $P$ is defined over $k$ if $P(V_k) \subset k$.

Recall that we assumed $G \subset \mathrm{GL}(V) \subset \mathrm{End}(V)$. We need to define precisely what we mean by the expression {\em integer points of $G$}.

\begin{definition}
The group $G$ is defined over $k$ if there exists a subset $\mathcal Q$ of the polynomial functions on $\mathrm{End}(V)$ defined over $k$ such that
\[
G = \{ g \in \mathrm{End}(V) \ \vert \ \forall P \in \mathcal Q, \ P(g) = 0 \}.
\]
\end{definition}

We now introduce a notion of integer points in $V_k$. This is well-defined using a $\Z$-lattice in $V_k$.

\begin{definition}
A $\Z$-lattice in $V_k$ is a finitely generated subgroup $\mathcal L$ of $V_k$ such that the natural map $\mathcal L \otimes_{\Z} k \to V_k$ is an isomorphism.
\end{definition}

We now assume that $k = \Q$. We fix a $\Z$-lattice $\mathcal L$ in $V_k$, and we define
\begin{equation}
G_\mathcal L = \{g \in G \ \vert \ g \mathcal L \subset \mathcal L \}.
\end{equation}
All these definitions lead to a way to construct arithmetic groups, thanks to the following proposition:

\begin{proposition}
The group $G_\mathcal{L}$ is an arithmetic subgroup of $G$.
\end{proposition}

We denote by $\mu$ the Haar measure of $G$.

\begin{definition}
A lattice of $G$ is a discrete subgroup $\Gamma$ such that $\Gamma \backslash G$ has finite volume with respect to the measure induced by $\mu$. If $\Gamma \backslash G$ is compact, the lattice $\Gamma$ is called a {\em cocompact} lattice in $G$.
\end{definition}

The Borel Harish-Chandra Theorem \cite[Theorems 1 and 3]{BHC} establishes a link between arithmetic groups and lattices of $\Q$-groups.

\begin{theorem}[Borel Harish-Chandra] \label{BHCTheorem}
Let $G$ be a complex algebraic linear group defined over $\Q$. If all rational characters of its identity component are trivial, then $G(\Z)$ is a lattice in $G(\R)$. If, moreover, every unipotent element of $G(\Z)$ lies in $R_u (G(\Q))$, then $G(\Z)$ is cocompact (this last property is known as Godement's criterion).
\end{theorem}

\subsection{The restriction of scalars} \label{restrictionSection}
We will frequently use restriction of scalars to construct arithmetic lattices. Therefore, we recall the basic facts about this construction.

Let $G$ be a semisimple linear algebraic group defined over an algebraic number field $k$, i.e. there exists a set $\mathcal Q$ of polynomials defined over $k$ such that
\[
G = \{ g \in \mathrm{End}(V) \ \vert \ \forall P \in \mathcal Q, \ P(g) = 0 \}.
\]
We assume that $k$ has $s$ real embeddings and $t$ complex embeddings. If $\sigma_1$ and $\sigma_2$ are two embeddings of $k$ into $\C$, we say that $\sigma_1$ and $\sigma_2$ are equivalent if $\sigma_1 = \overline{\sigma_2}$. Let $S^\infty$ be a complete set of representatives of equivalence classes of embeddings of $k$. For any $\sigma \in S^\infty$, we denote by $k^\sigma$ the completion of $\sigma(k)$, that is $\R$ or $\C$. We also denote by $\mathcal O_k$ the ring of algebraic integers of $k$.

For any $\sigma \in S^\infty$, we consider the linear algebraic $G^\sigma$ group defined over $\sigma(k)$ and given by
\begin{equation}
G^\sigma := \{ g \in \mathrm{End}(V) \ \vert \ \forall P \in \mathcal Q, \ \sigma(P)(g) = 0 \}.
\end{equation}

The main idea behind restriction of scalars is reminiscent of Dirichlet's unit theorem. The latter says that the image of the map $\mathcal O_k \to \R^{s+2t}, \ a \mapsto \{ \sigma(a) \}_{\sigma \in S^\infty}$ is a lattice in $\R^s \times \C^t \simeq \R^{s+2t}$. In particular, if we consider the $\Q$-form of $\R^{(s+2t)n}$ given by the image of
\begin{equation}
\Delta : k^n \to \R^{(s+2t)n}, \ (x_1, \ldots, x_n) \mapsto \{ (\sigma(x_1), \ldots, \sigma(x_n)) \}_{\sigma \in S^\infty},
\end{equation}
then, it admits $\Delta(\mathcal O_k^n)$ as a $\Z$-lattice.

The same idea applies to groups, yielding:

\begin{proposition}[Restriction of scalars] \cite[Proposition 5.4.8]{Mor} \label{restrictionscalars}
Let $G$ be a semisimple linear algebraic group defined over $k$. The group
\[
\mathrm{Res}_{k/\Q} (G) := \prod_{\sigma \in S^\infty} G^\sigma
\]
is defined over $\Q$ (for the $\Q$-form constructed as above on $\R^{s+2t}$), and its set of integer points is the image of the map
\[
\Delta : G(\mathcal O_k) \to \mathrm{Res}_{k/\Q} (G), \ g \mapsto \{\sigma(g)\}_{\sigma \in S^\infty}.
\]
\end{proposition}

Compact factors of algebraic groups will play a crucial role in this paper. Therefore, if at least one of the factors $G^\sigma$ in Proposition~\ref{restrictionscalars} is compact, it is a direct corollary of Godement's criterion that the integer points of $\mathrm{Res}_{k/\Q} (G)$ form a cocompact lattice in $\mathrm{Res}_{k/\Q} (G)$.

\section{Arithmetic data} \label{SecPre}

This section is devoted to the precise formulation of our question concerning GIB manifolds. We introduce the monodromy group of a GIB manifold, which naturally leads to the definition of GIB arithmetic data. The extraction of purely algebraic data from GIB manifolds allows us to understand the monodromy group independently of the geometric structure of these manifolds.

\subsection{Monodromy group of a GIB structure} \label{SecMonodromy} Let $M := \Gamma \backslash (\R^q \times N)$ be a GIB manifold (see Definition~\ref{GIBdef}). We denote by $P$ the projection of $\Gamma$ onto $\mathrm{Isom}(N, g_N)$. The identity component ${\bar P}^0$ of the closure $\bar P$ of $P$ is an abelian connected Lie group \cite[Corollary 2.6]{FZ2}. Moreover, since the projection sends $\Gamma$ isomorphic onto $P$, the normal subgroup $\Gamma_0 := \Gamma \cap (\mathrm{Sim} (\R^q) \times {\bar P}^0)$ of $\Gamma$ is abelian and is a lattice in $\R^q \times {\bar P}^0$ (see the paragraph following Corollary 2.6 in \cite{FZ2}).

Up to passing to the universal cover of $\R^q \times {\bar P}^0$, we may assume that $\R^q \times {\bar P}^0$ is $\R^n$ and $\Gamma_0$ is the canonical lattice $\Z^n$. The action of $\Gamma$ by conjugation on $\R^q \times {\bar P}^0$ preserves $\Gamma_0$, thus it induces a representation $\rho: \Gamma \to \mathrm{GL} (n,\Z)$ that preserves the decomposition $\R^q \times {\bar P}^0$.

\begin{definition} \label{MonGroup}
We call $\rho(\Gamma)$ the {\em monodromy group} of the GIB manifold, and $\rho$ the {\em monodromy representation}.
\end{definition}

The term monodromy in Definition~\ref{MonGroup} is justified as follows. The natural foliation induced by the submersion $\R^q \times N \to N$ descends to a foliation on the quotient $\Gamma \backslash (\R^q \times N)$. The closures of the leaves of this foliation are the images, under the canonical projection, of the sets $\R^q \times {\bar P}^0 x$ for $x \in N$ \cite[Lemma 4.11]{Kou}. After possibly replacing $M$ by a finite covering and $N$ by its orthonormal frame bundle, this foliation becomes a fibration $M \to X$ over a compact manifold, whose fibers are tori \cite[Theorem 1.6]{FZ2}. 

Since $\Gamma$ is a subgroup of $\mathrm{Sim}(\R^q) \times \mathrm{Isom}(N,g_N)$, there is a well-defined notion of orthogonality on the manifold $M$. Therefore, let $x \in M$ and let $F$ be the fiber of the fibration $M \to X$ containing $x$. Then $(T_x F)^\perp$ defines a horizontal space at $x$. The union of all the horizontal spaces defines a horizontal distribution on $M$, and hence an Ehresmann connection.

This connection lifts to the universal cover $\R^q \times N$. In the present setting, ${\bar P}^0$ acts freely on $N$, and since it is a closed subgroup of $\mathrm{Isom}(N,g_N)$, it also acts properly. We therefore obtain a fibration $\R^q \times N \to N/{\bar P}^0$ which is the natural lift of  the fibration $M \to X$. The lift of the horizontal distribution constructed before is given by the distribution orthogonal to the fibers.

Let $\bar x \in X$, and let $\bar c: S^1 \to X$ be a closed continuous loop based at $\bar x$. This curve lifts to a horizontal curve $\bar c' : [0,1] \to M$ in $M$. In turn, $\bar c'$ lifts to a horizontal curve $c$ in $\R^q \times N$. Since $x := c(0)$ and $y := c(1)$ project to the same fiber in $M$, there exist $\gamma \in \Gamma$ and $t_0 \in \R^q \times {\bar P}^0$ such that $y = \gamma t_0 \cdot x$. For any $t \in \R^q \times {\bar P}$, the curve $t \cdot c$ is still horizontal and joins $t \cdot x$ to
\[
t \cdot y = \gamma ((\gamma^{-1}t \gamma) t_0 \cdot x).
\]
Identifying the torus $(\R^q \times {\bar P}^0) / \Gamma_0$ with the fiber over $\bar x$ in $M$, we see that the holonomy of the connection on $M$ acts by transformations of the form
\begin{equation} \label{eq1}
(\R^q \times {\bar P}^0) / \Gamma_0 \to (\R^q \times {\bar P}^0) / \Gamma_0, \ t \mapsto (\gamma^{-1}t \gamma) t_0
\end{equation}
for some $\gamma \in \Gamma$ and $t_0 \in (\R^q \times {\bar P}^0) / \Gamma_0$. These maps are well defined because the action of $\Gamma$ by conjugation on $(\R^q \times {\bar P}^0)$ preserves $\Gamma_0$. Choosing a basis of $\Gamma_0$ to Identify the torus $(\R^q \times {\bar P}^0) / \Gamma_0$ with $\R^n / \Z^n$, the map \eqref{eq1} can be rewritten as
\begin{equation}
 \R^n / \Z^n \to \R^n / \Z^n, \ t \mapsto \rho(\gamma) t +  t_0.
\end{equation}
Thus, the image of $\rho$ is the group of linear parts of the holonomy transformations of the connection on $M$, which is precisely the monodromy of the connection. Note that the translation part $t_0$ gives information about the integrability of the horizontal distribution. Indeed, it is integrable if the subgroup translation parts is discrete.

\subsection{The arithmetic data of a GIB manifold} We keep notation from the previous section. We identify $\R^q \times {\bar P}^0$ with its universal cover $\R^n$ and view the monodromy group of the GIB manifold $M$ as a subgroup of $\mathrm{GL}(n, \Z)$. 

Proposition 4.9 in \cite{Fla25} implies that $\R^q$ is an {\em irrational subspace} of $\R^n$ in the following sense:

\begin{definition}
A vector subspace $V$ of $\R^n$ is irrational with respect to a lattice $\mathcal O$ if $V + \mathcal O$ is dense in $\R^n$. If there is no ambiguity, we omit the lattice and we simply say that $V$ is an irrational subspace of $\R^n$.
\end{definition}

We set $H := \R^q$. We also define $V$ as the universal cover of ${\bar P}^0$, and $b$ as the bilinear form induced on $\R^q$ by its Euclidean metric. Then, the following properties hold:
\begin{enumerate}[(a)]
\item $\R^n = H \oplus V$;
\item the monodromy group $\rho(\Gamma)$ preserves the decomposition $H \oplus V$;
\item \label{assumptions} the restriction of $\rho(\Gamma)$ to $H$ lies in the group of similarities of an inner product $b : H \times H \to \R$;
\item the restriction of $\rho(\Gamma)$ to $H$ contains a non-isometric similarity.
\end{enumerate}

\begin{remark}
The last property implies strong irrationality of $H$, i.e. $H \cap \Z^n = \{0 \}$. Indeed, $H \cap \Z^n$ is stable by $D$; in particular, it is stable by a non-isometric similarity $\gamma$ of $(H,b)$ which is the restriction of an element of $D$. Replacing $\gamma$ by $\gamma^{-1}$ if necessary, this similarity is contracting and thus cannot preserve a discrete set of integer points, unless this set is trivial.
\end{remark}

\begin{definition} \label{GIBdata}
We call the tuple $(n,H,V,b)$ the {\em arithmetic data} of the GIB manifold $M$. Conversely, let $n \ge 2$, and suppose we are given two subspaces $H$ and $V$ of $\R^n$ such that $\R^n = H \oplus V$ and $H$ is irrational with respect to the lattice $\Z^n$, together with a inner product $b$ on $H$. If there exists an element of $\mathrm{GL}(n, \Z)$ preserving the decomposition $H \oplus V$, and restricting to a strict similarity on $(H,b)$, we say that $(n,H,V,b)$ are {\em GIB arithmetic data}.
\end{definition}

In light of the discussion above, we endow GIB arithmetic data with a natural group of transformations.

\begin{definition}\label{AutGroup}
Let $(n,H,V,b)$ be GIB arithmetic data. The subgroup of all the matrices of $\mathrm{GL}(n,\Z)$ that preserve the decomposition $H \oplus V$ and restrict to similarities of $(H,b)$ is called the {\em automorphism group} of the GIB arithmetic data $(n,H,V,b)$, and it is denoted by $\mathrm{Aut}(n,H,V,b)$. A finitely generated subgroup of $\mathrm{Aut}(n,H,V,b)$ whose restriction to $H$ contains a non-isometric similarity of $b$ is called a {\em GIB monodromy subgroup} of $(n,H,V,b)$. Note that Theorem~\ref{structureThm} implies that $\mathrm{Aut}(n,H,V,b)$ is itself a GIB monodromy subgroup of $(n,H,V,b)$.
\end{definition}

Invoking Selberg's lemma, one may assume that the automorphism group of GIB arithmetic data has no torsion, up to taking a finite index subgroup. So far, no non-abelian subgroups of such an automorphism group have been constructed (see \cite{Fla24, Fla25}). We provide here an example where the automorphism group is not abelian, using restriction of scalars.

\begin{Ex} \label{FundEx}
We consider the quadratic form $q := x^2 + y^2 - \sqrt{5} z^2$ on $\R^3$ and its conjugate $q^* := x^2 + y^2 + \sqrt{5} z^2$. One has $\mathrm{SO}(q) \simeq \mathrm{SO(1,2)} \simeq \mathrm{PSL}(2, \R)$ and $G := \mathrm{SO}(q^*) \simeq \mathrm{SO} (3)$ is compact. The group $\mathrm{SO}(q)$ is defined over $k := \Q[\sqrt{5}]$. Applying restrictions of scalars, $\mathrm{Res}_{k/\Q} (G) = \mathrm{SO}(q^*) \times \mathrm{SO}(q)$ is defined over $\Q$. Moreover, in the notations of Section~\ref{restrictionSection}, the $\Q$-form of $\R^6$ is $\Delta(k^3)$, and the $\Z$-lattice is $\mathcal L := \Delta(\mathcal O_k^3)$. Consequently, the set $\Gamma_1$ of integer points of $\mathrm{Res}_{k/\Q} (G)$ is a cocompact lattice, and its projection onto $\mathrm{SO}(q)$ is a cocompact lattice of $\mathrm{PSL}(2, \R)$. In particular, $\Gamma_1$ is non-abelian.

The group $\Gamma_1$ preserves a decomposition $H \oplus V$ of $\R^6$, where $H$ has dimension $3$, corresponding to the splitting $\simeq \mathrm{SO}(q^*) \times \mathrm{SO}(q)$. Since the ring of algebraic integers $\mathcal O_k$ of $k$ is dense in $\R$, $\mathcal O_k^3$ is dense in $\R^3$, and $H$ is an irrational subspace of $\R^6$ with respect to $\mathcal L$.

Let $\lambda := \frac{-3 +\sqrt{5}}{2}$, which is a unit of $\Q [\sqrt{5}]$, and let $A$ be the matrix of the endomorphism acting by multiplication by $\lambda$ on $H$ and by multiplication by $\lambda^{-1}$ on $V$. The matrix $A$ also preserves the $\Z$-lattice $\mathcal L$ and the decomposition $H \oplus V$. Moreover, $A$ acts as a similarity of ratio $\lambda \neq \pm 1$ on $H$.

Altogether, $(6, H, V, q^*)$ are GIB algebraic data, and $\langle A \rangle \rtimes \Gamma_1$ has finite index in $\mathrm{Aut}(6, H, V, q^*)$, which is therefore non-abelian.
\end{Ex}

\section{The Zariski closure of the automorphism group: proof of Theorem~\ref{structureThm}} \label{SecZarclos}

In this section, we prove Theorem~\ref{structureThm} and we discuss its implications. More precisely, we investigate the constraints on the semisimple part $S$ introduced in the theorem.

\subsection{Decomposition of the automorphism group} Let $(n, H, V, b)$ be GIB arithmetic data as defined in Definition~\ref{GIBdata}. Let $D$ be a subgroup of $\mathrm{Aut}(n,H,V,b)$. We aim to give a description of the group $D \subset \mathrm{GL}(n, \Z)$. In order to do so, we consider $D$ as a discrete subgroup of a larger Lie group, namely its Zariski closure in $\mathrm{GL}(n, \R)$, and we use results on arithmetic groups to prove Theorem~\ref{structureThm}.

We consider the set of polynomials
\begin{equation}
I(D) := \{ P \in \R[X] \ \vert \ \forall A \in D, \ P(A) = 0 \}.
\end{equation}
The space $I(D)$ is a vector subspace of $\R[X]$, and we claim that there exists a finite set of elements in $\Q[X]$ generating $I(D)$. Indeed, for any $k \in \N$, let
\begin{equation}
I_k(D) := \{ P \in I(D) \ \vert \ \mathrm{deg}(P) \le k \}.
\end{equation}
Then, the space $I_k(D)$ is exactly the subspace of $\R_k[X]$ (the space of polynomials with real coefficients and degree less that $k$) satisfying the equations $P(A) = 0$ for all $A \in D$. Consequently, the coefficients of $P \in I_k(D)$ are constrained by a set of linear equations with integer coefficients, so they can be solved in $\Q[X]$. Now, the $\Q$-dimension of $I_k(D) \cap \Q[X]$ is equal to the $\R$-dimension of $I_k(D)$, and we conclude that $I_k(D) \cap \Q[X]$ is a $\Q$-form of $I_k(D)$.

We now look at the Zariski closure of $D$ in $\mathrm{GL}(n,\R)$, i.e.
\begin{equation}
G := \{ A \in \mathrm{GL}(n, \R) \ \vert \ \forall P \in I(D), \ P(A) = 0 \}. 
\end{equation}
The set $G$ is a linear algebraic group, therefore it is a Lie subgroup of $\mathrm{GL}(n, \R)$. Moreover, $G$ is defined over $\Q$ by the previous discussion. Up to passing to a finite-index subgroup of $D$, we may assume that $G$ is connected in the Zariski topology.

We claim that $G$ is a reductive group. In order to prove this, we consider the unipotent radical $R_u(G)$ of $G$ and we aim to show that this group is trivial. Since $G$ is defined over $\Q$, and $\Q$ is a perfect field, $R_u (G)$ is defined over $\Q$ because it is $\Q$-closed. By definition, $R_u(G)$ is connected, so we can apply \cite[Corollary 18.3]{Borel} to the perfect field $\Q$ and the connected linear algebraic group $R_u(G)$ to conclude that $R_u(G)(\Q)$ (i.e. the rational points of the unipotent radical of $G$) is Zariski-dense in $R_u(G)$. Thus, to prove that $R_u(G)$ is trivial, it is sufficient to prove that $R_u(G)(\Q)$ is trivial.

Stabilizing a subspace of $\R^n$ amounts to satisfying polynomial equations. Therefore, the group $G(\Q)$, which is included in the Zariski closure of $D$ in $\mathrm{GL}(n, \R)$, still preserves the decomposition $H \oplus V$. Moreover, being a similarity once restricted to $(H, b)$ is also a polynomial condition, thus the restriction of $G$ to $V$ consists of $b$-similarities. For any element $A$ of $R_u(G)(\Q)$, since $A$ is unipotent, its restriction to $H$ is a unipotent similarity of $(H,b)$, thus it is the identity. The conclusion then follows from the following lemma:

\begin{lemma}
Let $A$ be an $n \times n$ matrix with coefficients in $\Q$ preserving the decomposition $H \oplus V$. If the restriction of $A$ to $H$ is the identity, then $A = I_n$.
\end{lemma}
\begin{proof}
Let $A \in \mathrm{M}_n(\Q)$ whose restriction to $H$ is the identity. There exists $p \in \N^*$ such that $p A \in \mathrm{M}_n(\Z)$, and $B := p(A - I_n)$ restricted to $H$ is the zero endomorphism. Since $B$ has coefficients in $\Z$, it induces a group homomorphism $B_T : T^n \to T^n$, where $T^n := \R^n / \Z^n$. But $B_T$ coincides with $p I_n$ on the image of $H$ in $T^n$, which is dense in $T^n$ because $H$ is irrational. This implies that $B = p I_n$, yielding $A = I_n$.
\end{proof}

As a consequence of this lemma, $A = I_n$ and $R_u(G)(\Q)$ is trivial, thus $G$ is reductive. In particular, $G$ is the almost direct product of the connected component of its center $T$ and of its commutator subgroup $S := [G,G]$ \cite[Proposition 2.2]{BoTi}. In addition, $S$ is semisimple, and the first corollary in \cite[Section 2.3]{Borel} implies that $S$ is a closed Zariski connected subgroup of $\mathrm{GL}(n, \R)$ defined over $\Q$, because $G$ is Zariski connected. We can now apply the Borel Harish-Chandra Theorem \ref{BHCTheorem}. Since $G(\Q)$, and thus $S (\Q)$, has no non-trivial unipotent element, and $S$ is semisimple, we conclude that $D_S := S (\Z)$ is a cocompact lattice of $S$.

\begin{remark}
There are several possibilities for $S$. First, if $G$ is abelian, then $S$ is trivial. The group $S$ could also be compact. In this case, the elements of $D_S$ are torsion elements, and we can avoid this situation by taking a finite-index subgroup of $D$, thanks to Selberg's Lemma. The remaining case is when $S$ is non-compact. In this latter situation, the lattice $D_S$ is infinite and non-abelian, because $S$ is semisimple.
\end{remark}

We first recall that $G$ is isogenous to $T \times S$. The restrictions to $H$ of the elements of $S$ are $b$-isometries because $S$ is a group of commutators. Consequently, if $D$ contains a matrix that restricts to a strict similarity of $(H,b)$, the group $T$ cannot be compact.

We claim that:
\begin{lemma} \label{normalizer}
The normalizer $\mathrm{Nor}_S (G(\Z))$ of $G(\Z)$ inside $S$ is discrete.
\end{lemma}
\begin{proof}
Let $(\gamma_1, \ldots, \gamma_\ell)$ be elements of $D$ such that $\overline{\langle \gamma_1, \ldots, \gamma_\ell \rangle}^\mathrm{Zar} = G$. This tuple exists because $G$ is finite-dimensional and a linear algebraic group has finitely many connected components.

For any $1 \le i \le \ell$ there exists a small neighborhood $U_i$ of the identity in $S$ such that for all $s' \in U_i$, $s'\gamma_i s'^{-1} = \gamma_i$ because $D_S$ is discrete and conjugation is continuous. Taking $U := \bigcap_{1 \le i \le \ell} U_i$, the open set $U$ is a small neighbourhood of the identity in $S$ such that $\mathrm{Nor}_S (G(\Z)) \cap U$ centralizes $\langle \gamma_1, \ldots, \gamma_\ell \rangle$, and thus $G$. After possibly shrinking $U$, we may assume that the intersection of $U$ with the center of $S$ contains only the identity, so $\mathrm{Nor}_S (G(\Z)) \cap U$ contains only the identity. This proves the lemma.
\end{proof}

Applying Lemma~\ref{normalizer} and using the fact that $S / D_S$ is compact, the image of $\mathrm{Nor}_S (G(\Z))$ in $S / D_S$ is finite. Hence, there exists $k > 0$ such that for any $a \in G(\Z) \subset \mathrm{Nor}_S (G(\Z))$, writing $a = st$ with $s \in S$ and $t \in T$, one has $s^k \in D_S$. This yields $a^k = s^k t^k$, and $s^k \in \mathrm{GL}(n,\Z)$, thus $t^k \in \mathrm{GL}(n,\Z)$. We deduce that $D_T D_S$ has finite index in $G(\Z)$. Since $D_S$ is a lattice in a semisimple Lie group, it is finitely presented. The group $D_T$ is a discrete subgroup of an abelian Lie group, thus it is finitely presented too. Therefore, $\Gamma := D_T D_S$ is finitely presented, and hence so is $G(\Z)$.

We now prove that $D_T$ is a lattice in $T$. We first show that $D_T$ is Zariski dense in $T$. Since the group $D_T D_S$ has finite index in $\Gamma$, the Zariski closure of $D_T D_S$ has finite index in $G$. Moreover, $G$ is Zariski connected by assumption, thus $D_T D_S$ is Zariski dense in $G$. We also know that $D_S$ is a lattice in the semisimple group $S$, thus $D_S$ is Zariski dense in $S$. The Lie group $G$ is isogenous to the product $T \times S$ and $D_T D_S$ is Zariski dense in $G$. It follows that $D_T$ is Zariski dense in $T$.

We are now in a position to prove that $T$ is $\Q$-anisotropic, i.e. it admits no non-trivial rational character defined over $\Q$. This will allow us to apply the Borel Harish-Chandra theorem. Let $\chi$ be a rational character defined over $\Q$ on $T$. By \cite[8.2 (c) and 8.5]{Borel}, $T$ is conjugate to a diagonal group over $\C$, and, in this basis of diagonalization, the character $\chi$ is of the form
\begin{equation}
\chi = \prod_{i=1}^n \chi_i^{m_i},
\end{equation}
where $\chi_i$ is the $i$-th diagonal entry, and $m_i \in \Z$. For any $g \in D_T$, the matrix $g$ has integer coefficients, so $\chi (g)$ lies in $\Q$ and is a product of eigenvalues of $g$ and their inverses. Since the eigenvalues of $g$ are algebraic units, $\chi(g)$ is an algebraic unit lying in $\Q$, thus it is $\pm 1$. It follows that $\chi^2$ is trivial. By \cite[8.2 (c)]{Borel}, and because $D_T$ is Zariski dense in $T$, the algebraic torus $T$ is the intersection of the kernels of the characters of $T$ vanishing on $D_T$. In particular, $T \subset \mathrm{ker} (\chi^2)$, and $\chi^2$ is trivial. Hence, $\chi$ has value $\pm 1$ on $T$. Since $\chi$ is continuous with respect to the Zariski topology and has value $1$ at the identity matrix, it is identically	 equal to $1$ on the Zariski connected group $T$. We conclude that $T$ is $\Q$-anisotropic, and, by the Borel Harish-Chandra theorem \ref{BHCTheorem}, $T(\Z) = D_T$ is a lattice in $T$.

We proved the following theorem:

\begin{theorem} \label{structureThm0}
Let $(n, H, V, b)$ be GIB arithmetic data and let $D$ be a subgroup of $\mathrm{Aut}(n, H, V, b)$. Then,
\begin{enumerate}
\item the Zariski closure of $D$ is commensurable to a reductive $\Q$-group $G$ isogenous to $T \times S$, where $T$ is the center of $G$ and $S$ is its commutator group.
\item if $D$ contains a matrix that restricts to a non-isometric similarity of $(H,b)$, then $T$ is non-compact.
\item the group $T$ is an anisotropic torus, and $D_T := T(\Z)$ is a lattice in $T$.
\item the set $D_S := S(\Z)$ is a cocompact arithmetic subgroup in the semisimple linear algebraic $\Q$-group $S$.
\item the group $\Gamma := D_T \times D_S$ has finite index in $G(\Z)$. In particular, $\Gamma$ and $G(\Z)$ are finitely presented.
\end{enumerate}
\end{theorem}

 Theorem~\ref{structureThm} is then a direct consequence of Theorem~\ref{structureThm0}.

\subsection{Constraints on the semisimple part} Let $(n,H,V,b)$ be GIB arithmetic data. Let $D$ be a GIB monodromy subgroup of $(n,H,V,b)$. If $G$ denotes the Zariski closure of $D$, then Theorem~\ref{structureThm0} implies that $G(\Z)$ is a GIB monodromy subgroup of $(n,H,V,b)$ as well. Moreover, $G$ is a reductive group, and $[G,G]$ is a semisimple algebraic group.

It is therefore natural to ask whether a given semisimple subgroup $S$ of $\mathrm{SL}(m,\R)$ ($m \ge 1$) defined over $\Q$ could arise as $[G, G]$ for suitable GIB arithmetic data. We recall that any semisimple Lie subgroup of $\mathrm{SL}(m,\R)$ is almost Zariski closed \cite[Theorem A4.9]{Mor}.

In this section, we prove that, up to adding a compact factor, any semisimple subgroup $S$ may be realized as the commutator subgroup of $G$. The additional compact factor is mandatory, because $S$ should preserve the decomposition $H \oplus V$, and its projection onto $\mathrm{GL}(H)$ is compact. In addition, $H$ is irrational, so this projection restricted to $S(\Z)$ is actually injective.

We use unitary groups, as described in \cite[Example 6.3.2]{Mor}, to prove our claim.

We fix a totally real number field $k$ and we denote by $S^\infty$ its set of archimedean places. Let $a > 0$, $b > 0$ be elements of $k$ such that $\sigma(a) < 0$ and $\sigma(b) < 0$ for all non-trivial $\sigma \in S^\infty$. Let $K$ be the quadratic extension of $k$ defined by $K := k[\sqrt{a}]$, and let $\tau$ be the non-trivial element of $\mathrm{Gal}(K/k)$. We denote by $\mathcal O_K$ the ring of integers of $K$, and we consider the $m \times m$ matrix
\begin{equation}
A := \mathrm{Diag} (b, \ldots, b, -1).
\end{equation}

The semisimple group $S$ is embedded into $\mathrm{SL}(2m,\R)$ via the map
\begin{equation}
\phi : \mathrm{SL}(m,\R) \to \mathrm{SL}(m,\R) \times \mathrm{SL}(m,\R), g \mapsto (g, (g^T)^{-1}).
\end{equation}
We consider the unitary group
\begin{equation}
\mathrm{SU} (A, \tau; \mathcal O_K) := \{ g \in \mathrm{SL}(m, \mathcal O_K) \ \vert \ \tau(g^T) A g = A \}.
\end{equation}

We define a $k$-form of $\R^{2m}$ by setting $W := \{ (s, A \tau(s)), \ s \in K^{m} \}$. One verifies that $W$ is a $k$-vector space and that the natural map $W \otimes_k \R \to \R^{2m}$ is an isomorphism.

We claim that $\phi(S)$  is defined over $k$ with respect to the $k$-form $E_k$ induced by $W$ on $\mathrm{End}(\R^{m}) \times \mathrm{End}(\R^{m})$. To prove this, first remark that
\begin{equation}\label{Fform}
\begin{aligned}
E_k &= \{ (B, C) \in \mathrm{End}(\R^{m}) \times \mathrm{End}(\R^{m}) \ \vert \ \ B(V) \subset V \} \\
&= \{ (B, A \tau(B) A^{-1}) \ \vert \ \ B \in \mathrm{Mat} (m, K) \}.
\end{aligned}
\end{equation}
Let $\mathcal Q$ be the subset of $\Q[x_1, \ldots, x_m]$ defining the algebraic group $S$, which exists because $S$ is defined over $\Q$. Any element $B$ of $\mathrm{SL}(m, K)$ can be written as $B =: M + \sqrt{a} N$, where $M$ and $N$ are matrices with coefficients in $k$. Consequently, for any polynomial $P \in \mathcal Q$, the equation
\[
0 = P(B) = P(M+\sqrt{a} N)
\]
splits into a system of polynomial equations with coefficients in $k$, where the entries are the coefficients of $M$ and $N$. Moreover, if $(B, C) := (B, A \tau(B) A^{-1}) \in E_k$ with $B = M + \sqrt{a} N$ as above, one has
\begin{align*}
2M = B+ A^{-1} C A && \text{ and } && 2 \sqrt{a} N = B - A^{-1} C A,
\end{align*}
so the entries of $M$ and $N$ are polynomials in the entries of $(B,C)$. Altogether, $\mathcal Q$ induces a set of polynomial equations with coefficients in $k$ on $E_k$. We add to this set the equation $C B^T = I_m$, and we denote by $\mathcal Q'$ the set of all polynomial equations obtained this way.

Consider the algebraic variety
\begin{equation}
\mathrm{Var}(\mathcal Q') = \{ (g,g') \in \mathrm{SL}(m,\R) \times \mathrm{SL}(m,\R) \ \vert \ P(g,g') = 0 \text{ for all $P \in \mathcal Q'$}\}.
\end{equation}
By construction, one has $\mathrm{Var}(\mathcal Q') = \phi(S)$, and we conclude that $\phi(S)$ is defined over $k$.

We consider the subgroup of $V$ defined by $\mathcal L := \{ (s, A \tau(s)), \ s \in \mathcal O_K^{m} \}$. It  is the set of $\mathcal O_k$-points of $V$, where we recall that $\mathcal O_k$ is the ring of algebraic integers of $k$. The subgroup of $\phi(S)$ defined by
\begin{equation}
G := \{ \phi(g) \ \vert \ \ g \in S, \ \phi(g) \mathcal L \subset \mathcal L \}
\end{equation}
is the group of $\mathcal O_K$-points of $\phi(S)$. From Equation~\eqref{Fform}, any element $(g,(g^T)^{-1})$ of $\phi(S)$ with $g \in \mathrm{SL}(m, \mathcal O_K)$ satisfies $\tau((g^T)^{- 1}) = A^{-1} g A$, and thus $\tau(g^T) A g = A$. Therefore,
\[
G = \phi(S \cap \mathrm{SU} (A, \tau; \mathcal O_K)).
\]

We now use restriction of scalars to conclude that
\[
D_S := \prod\limits_{\sigma \in S^\infty} \sigma(\phi(S \cap \mathrm{SU} (A, \tau; \mathcal O_K)))
\]
is an arithmetic subgroup of
\[
S' := \mathrm{Res}_{k/\Q} (\phi(S)) = \prod\limits_{\sigma \in S^\infty} \phi(S)^\sigma.
\]
Let $\sigma \in S^\infty$ be a non-trivial place. Since $\sigma(a) < 0$ and $\sigma(b) < 0$, one has
\[
\sigma(\phi(S \cap \mathrm{SU} (A, \tau; \mathcal O_K))) \subset \phi(\mathrm{SU} (m)),
\]
and this last group is compact. Hence, $D_S$ is cocompact in $S'$ by Godement's criterion, and it is conjugate to a subgroup of $\mathrm{GL}(n := \vert S^\infty \vert 2m, \Z)$. In addition, Borel's density theorem implies that the Zariski closure of $D_S$ is isomorphic to $S$, up to a compact factor.

The group $D_S$ acts naturally on $\R^n \simeq \prod_{\sigma \in S^\infty} (k^\sigma)^{2m}$ (recall that $k^\sigma$ is the completion of $\sigma(k)$), where the factors $(k^\sigma)^{2m}$ are all isomorphic to $\R^{2m}$. We define the subspace $H$ of $\R^n$ as one of the factors corresponding to a non-trivial place, and $V$ is the product of all the other factors. In particular, $D_S$ restricts to a relatively compact subgroup of $\mathrm{GL}(H)$ on $H$. Therefore, this restriction preserves an inner product $b$.

It remains to find a matrix commuting with $D_S$ and acting as a strict similarity on $H$. This is easily done by choosing a non-trivial unit $\alpha \in k$, and considering the matrix $A_0$ acting by multiplication by $\sigma(\alpha)$ on each $(k^\sigma)^{2m}$, for all $\sigma \in S^\infty$. This matrix and its inverse preserve the $\Z$-lattice $\{ \{\sigma(x)\}_{\sigma \in S^\infty} \ \vert \ x \in \mathcal L\}$, and they restrict to similarities of $(H,b)$. Since $k$ is totally real, this similarity is not an isometry.

The group $\langle A_0 \rangle \times D_S$ is a GIB monodromy subgroup of $(n,H,V,b)$, and its Zariski closure is the product $T \times S'$, where $T$, the Zariski closure of $\langle A_0 \rangle$, is a torus.

\subsection{Constraints on the arithmetic group} The discussion of the previous section shows that any semisimple subgroup $S$ of $\mathrm{SL}(m, \R)$ can occur in Theorem~\ref{structureThm0}, up to adjoining a compact factor. A natural follow-up question is whether any arithmetic subgroup of a semisimple group can occur as the $D_S$ component (in the notations of Theorem~\ref{structureThm0}) of a GIB monodromy subgroup of suitable GIB arithmetic data.

Let $S$ be a semisimple algebraic subgroup of $\mathrm{SL}(n, \R)$ defined over $\Q$. By \cite[Proposition 5.5.12]{Mor} and the proof therein (see also \cite[6.21 (ii)]{BoTi}), $S$ is isogenous to a product of semisimple groups $S_1 \times \ldots \times S_m$, where each $S_i$ is defined over $\Q$, and such that $S_i(\Z)$ is irreducible as a lattice in $S_i$. Let $S_{i,1} \times \ldots \times S_{i,r_i}$ be the decomposition of $S_i$ into simple groups. There exists an algebraic number field $k_i$ of degree $r_i$, and set of archimedean places $(S^\infty)_i$, such that $S_{i,1}$ is an algebraic subgroup of $\mathrm{SL}(n_i, k^\sigma)$ defined over $k_i$, and $S_i$ is isogenous to
\begin{equation}
\prod_{\sigma \in (S^\infty)_i} (S_{i,1})^{\sigma}.
\end{equation}
In this case, the arithmetic group $S_i(\Z)$ is the image of the $\mathcal O_{k_i}$-points of $S_i$ by the map
\begin{equation}
\Delta_i : g \mapsto \{ \sigma(g) \}_{\sigma \in (S^\infty)_i}.
\end{equation}

The group $S(\Z)$ should preserve a decomposition $H \oplus V$, as well as an inner product $b$ on $H$. Therefore, it projects injectively into a compact subgroup of $\mathrm{GL}(H)$. Consequently, each $S_i$ must admit a compact factor, and we may assume that this factor is $S_{i,1}$. By construction, $S_i(\Z)$ projects injectively and densely into $S_{i,1}$.

Summarizing, $S$ is isogenous to
\begin{equation}
\prod_{i=1}^m \prod_{\sigma \in (S^\infty)_i} (S_{i,1})^{\sigma}.
\end{equation}
We may discard the simple groups $S_i$ in this decomposition. Indeed, assume that $S_i$ is simple. Since $S_i$ preserves the decomposition $H \oplus V$ and projects injectively into a compact subgroup of $\mathrm{GL}(H)$, the whole group $S_i$ is compact. It follows that $S_i(\Z)$ is finite, and we may eliminate it using Selberg's lemma.

We now ask whether $S$ can be realized as the semisimple part of a GIB monodromy subgroup of GIB arithmetic data. We first observe that each semisimple group $S_i$ has a natural faithful representation $\rho_i$ on the vector space
\begin{equation}
E_i := \bigoplus_{\sigma \in (S^\infty)_i} (k_i^\sigma)^{n_i}.
\end{equation}
Hence, we define a representation of $S$ by taking the external tensor product
\begin{equation}
\rho := \rho_1 \boxtimes \ldots \boxtimes \rho_m.
\end{equation}
We define two subspaces of $\bigotimes_{i=1}^m E_i$ by
\begin{align}
H := \bigotimes_{i=1}^m (k_i^\mathrm{id})^{n_i} && \text{and} && V := \bigoplus_{\substack{(\sigma_1, \ldots, \sigma_m) \in \prod_{i=1}^m (S^\infty)_i \\ (\sigma_1, \ldots, \sigma_m) \neq (\mathrm{id}, \ldots, \mathrm{id})}} \ \bigotimes_{i=1}^m (k_i^{\sigma_i})^{n_i}.
\end{align}
Since $S$ projects as a compact subgroup of $\mathrm{GL}(H)$, there exists an inner product $b$ on $H$ preserved by this projection.

It remains to find a matrix whose restriction to $(H,b)$ is a non-isometric similarity. None of the factor $S_i$ is simple, therefore we can apply the following lemma:

\begin{lemma}
Let $k$ be an algebraic number field and $S^\infty$ its set of archimedean places. If $\vert S^\infty \vert \ge 2$, then there exists a unit $\alpha$ of $k$ with $\vert \alpha \vert \neq 1$.
\end{lemma}
\begin{proof}
Let $\sigma_1, \ldots, \sigma_s$ be the real places of $k$ and let $\sigma_{s+1}, \ldots, \sigma_{s+t}$ be the complex ones. We assume that all units of $k$ have complex modulus equal to $1$, and we prove the lemma by contraposition.

By Dirichlet's units theorem, the inclusion
\[
\Delta : \mathcal O^\times_k \to \R^{s+t}, \ u \mapsto (\ln \vert \sigma_1(u) \vert, \ldots, \ln \vert \sigma_{s+t}(u) \vert)
\]
is a lattice in the hyperplane of $\R^{s+t}$ defined in the canonical coordinates by
\[
x_1 + \ldots + x_s + 2 x_{s+1} + \ldots + 2 x_{s+t} = 0.
\]
In particular, $\Delta (\mathcal O^\times_k)$ has rank $s+t-1$. The function $\ln \vert \cdot \vert$ vanishes on $\mathcal O^\times_k$, and we may assume that $\sigma_1 = \mathrm{id}$ (the complex case is treated in the same way). Then, $\Delta (\mathcal O^\times_k)$ is a discrete subgroup of the space defined by the equation
\[
x_2 + \ldots + x_s + 2 x_{s+1} + \ldots + 2 x_{s+t} = 0,
\]
and its rank is at most $s + t - 2$. This is only possible if $s + t = 1$, i.e. $s = 1$ or $t = 1$, which amounts to saying $\vert S^\infty \vert = 1$.
\end{proof}

This lemma implies that there exists $\alpha \in k_1$ with $\vert \alpha \vert \neq 1$. We now consider the endomorphism $A_0$ of $E_1 = \bigoplus_{\sigma \in (S^\infty)_1} (k_1^\sigma)^{n_1}$ acting by multiplication by $\sigma(\alpha)$ on $(k_1^\sigma)^{n_1}$, for all $\sigma \in (S^\infty)_1$. The group $\langle A_0 \rangle$ preserves the canonical $\Z$-lattice defined in Section~\ref{restrictionSection}. In addition, $A_0$ acts naturally on $\bigotimes_{i=1}^m E_i$ by multiplication on the first factor of the tensor product. Hence, $\langle A_0 \rangle \times S(\Z)$ is a GIB monodromy subgroup of $(n,H,V,b)$ for a suitable integer $n$, and its Zariski closure is $T \times S$, where $T$ is the Zariski closure of $\langle A_0 \rangle$.

The discussion of this section is summarized in the following proposition.

\begin{proposition}
The groups $D_S$ appearing in Theorem~\ref{structureThm0} are exactly, up to taking a finite-index subgroup, the products of arithmetic groups of the form $\mathrm{Res}_{k/\Q}(G)(\Z)$, where $G \subset \mathrm{SL}(m,\R)$ is defined over the algebraic number field $k$, and one of the factors of $\mathrm{Res}_{k/\Q}(G)(\Z)$ (see Proposition~\ref{restrictionscalars}) is compact.
\end{proposition}

\section{From GIB monodromy subgroups to GIB manifolds: proof of Theorem~\ref{GIBmonodromy}} \label{FromtoGIB}

We return to our original motivation, namely to find GIB manifolds with large monodromy groups, and we prove Theorem~\ref{GIBmonodromy}. To construct examples, we first recall that GIB manifolds are equivalent to LCP manifolds \cite[Theorem 1.9]{FZ2}. This equivalence allows us to use a construction for LCP manifolds described in \cite{Fla25}, which we outline below for the convenience of the reader.

Let $(n,H,V,b)$ be GIB arithmetic data and let $D$ be a GIB monodromy subgroup of $(n,H,V,b)$. Let $X$ be a compact manifold together with a surjective group homomorphism $\varphi : \pi_1(X) \to D$. We denote by $\tilde X$ the universal cover of $X$. 

Consider the group of diffeomorphisms of $\R^n \times \tilde X$ defined by
\begin{equation}
\Gamma := \Z^n \rtimes \{ \R^n \times X \ni (a,x) \mapsto (\varphi(\gamma) a, \gamma(x)) \ \vert \ \gamma \in \pi_1(X) \}.
\end{equation}
Then the quotient manifold $\Gamma \backslash (\R^n \times \tilde X)$ is compact and carries a GIB structure. Moreover, the monodromy group of this GIB manifold is $D$.

\begin{example}
We return to Example~\ref{FundEx}. In this situation, the group $D := \langle A \rangle \times \Gamma_1$ can be realized, up to finite index, as the fundamental group of $S^1 \times C$ where $C$ is a compact hyperbolic surface. Indeed, $\Gamma_1 \subset \mathrm{PSL}(2,\R)$ (the isometry group of $\HH^2$), and $\Gamma_1$ is a lattice in $\mathrm{PSL}(2,\R)$. It suffices to remove the torsion elements of $\Gamma$, so that it acts freely on $\HH^2$. This can be achieved using Selberg's lemma.
\end{example}

Finding a GIB manifold with monodromy group $D$ is thus tantamount to finding a compact manifold $X$ together with a surjective morphism $\rho_0 : \pi_1(X) \to D$. Since $D$ is finitely generated (see Definition~\ref{AutGroup}), it is a quotient of a finitely generated free group $F$. This yields a surjective morphism $\varphi : F \to D$. Finally, since any finitely presented group can be realized as the fundamental group of a compact $4$-manifold, this proves Theorem~\ref{GIBmonodromy}.

The reader may prefer a more explicit description of the compact manifold $X$. We therefore provide a concrete construction and we prove Proposition~\ref{ExplicitMon}. Let $X$ be a closed surface of genus $g$, and assume that $D$ is generated by $g$ elements. There exists a surjective group morphism from $\pi_1(X)$ onto the free group on $g$ generators, and hence a surjective homomorphism from $\pi_1(X)$ to $D$.

By Theorem~\ref{structureThm0}, $\overline D^\mathrm{Zar}$ is an almost direct product $T \times S$ where $T$ is a torus and $S$ is semisimple. If $D$ is cocompact in its Zariski closure, as it is the case for $\mathrm{Aut}(n,H,V,b)$, we define $\tilde X$ as the symmetric space associated to $T \times S$, which is the quotient of $T \times S$ by the maximal compact factor of $S$, and we take $X := \tilde X / D$. Using Selberg's lemma, we may assume that $D$ has no torsion elements, thus $X$ is a genuine manifold and $\pi_1(X) = D$.

This discussion proves Proposition~\ref{ExplicitMon}.

\section{GIB manifolds and symmetric spaces} \label{SecSym}

This section is devoted to the study of GIB manifolds $\Gamma \backslash (\R^q \times N)$, where $N$ is a symmetric space of non-compact type. In this situation, we describe the Zariski closure of the group $\bar P$ in $\mathrm{Isom}(N)$ (see Section~\ref{SecMonodromy}) and prove that ${\bar P}^0$ is the unipotent radical of this group. This allows us to identify the monodromy subgroup of the GIB manifold and construct new examples of GIB manifolds.

\subsection{Minimal cocompact algebraic subgroups}
We require some preliminary results on cocompact subgroups of semisimple Lie groups. Most of the material of this section is well known, but we provide a self-contained discussion since we were unable to find a complete reference.

A real Lie group is called algebraic if it is a finite index subgroup of the real points of a complex algebraic group defined over $\R$.

Throughout this section, $G$ is a semisimple connected real Lie group with finite center, and $\mathfrak{g}$ is its Lie algebra. In particular, $G$ is isogenous to a linear algebraic subgroup of $\mathrm{GL}(\mathfrak{g})$ via the adjoint representation. Let $H$ be a connected, cocompact, algebraic subgroup of $G$.

Let $\theta$ be a Cartan involution on $G$, and let $\mathfrak g = \mathfrak k \oplus \mathfrak p$ be the induced Cartan decomposition. These data also allow us to define an Iwasawa decomposition $G = KAN$, or equivalently $\mathfrak g = \mathfrak k \oplus \mathfrak a \oplus \mathfrak n$, where $\mathfrak n = \bigoplus_{\alpha \in \Sigma^+} \mathfrak g_\alpha$ with $\Sigma^+$ a positive root system for the action of $\mathfrak a$ on $\mathfrak g$.

We call a  closed subgroup of $G$ a {\em minimal algebraic cocompact subgroup} if it is cocompact and does not contain a strictly smaller algebraic cocompact subgroup of $G$. The minimal algebraic cocompact subgroups can be described using the Iwasawa decomposition of $G$. This is a direct consequence of the proof of \cite[Lemma 3.2]{BaNe}, but we give some details here for the convenience of the reader.

\begin{lemma} \label{AN}
The group $H$ contains a conjugate of $AN$.
\end{lemma}
\begin{proof}
The group $A$ is a maximal $\R$-split torus of $G$, and $N$ is the subgroup with Lie algebra generated by the positive root spaces of $G$ for the action of $A$. Consequently, $AN$ is a connected, solvable, algebraic, $\R$-split subgroup of $G$. Since $G/H$ is compact and $AN$ acts algebraically on $G/H$, we can apply \cite[Lemma 3.1]{BaNe} to show that this action has a fixed point. It follows that $AN$ is conjugate to a subgroup of $H$.
\end{proof}

Lemma~\ref{AN} has the following consequence:

\begin{corollary} \label{Commute}
The centralizer $C_G(H)$ of $H$ inside $G$ is an almost direct product between $C(G)$ and a compact factor of $G$.
\end{corollary}
\begin{proof}
By Lemma~\ref{AN}, we may assume, up to conjugation, that $H$ contains the subgroup $AN$ of $G$, where $KAN$ is the Iwasawa decomposition of $G$. In particular, the elements of $C_G(H)$ commute with $AN$. Therefore, they must be contained in $K$ by the general theory of root spaces in semisimple Lie groups.

We claim that $C_K(AN)$ is exactly the set of elements $k \in K$ acting trivially on $K \backslash G$ with respect to the right action. Indeed, if $k \in K$ commutes with $AN$, for any $(a,n) \in A \times N$, one has
\[
K a n k = K k a n = K a n,
\]
thus $k$ acts trivially on $K \backslash G$. Conversely, if $k \in K$ acts trivially on $K \backslash G$, then for any $a \in A$ one has $K a k = K a$, and $a k a^{-1} \in K$. But one has $[\mathfrak a, \mathfrak k] \subset \mathfrak p$, so taking $a$ small enough, we see that $a k a^{-1} = k$, and this remains valid for any $a \in A$. Hence, $k$ commutes with $A$. In particular, the conjugation by $k$ preserves the positive root spaces, and it preserves the group $N$. By the same argument as above, one has, for any $n \in N$, $n k n^{-1} \in K$, so $n k n^{-1} k^{-1} \in K$, but we also have $n k n^{-1} k^{-1} \in N$ and $N \cap K = \{ \mathrm{id} \}$. We deduce that $k$ commutes with $N$.

Now, for any $g \in G$ and $k \in C_K(AN)$, and for any $h \in G$:
\[
K h g k g^{-1} = K h g g^{-1} = K h,
\]
therefore $C_K(AN)$ is a compact normal subgroup of $G$. It is also closed by definition, and this concludes the proof because $G$ has finite center.
\end{proof}

We now prove that $AN$ is a maximal $\R$-split solvable algebraic subgroup of $G$.

\begin{lemma} \label{ANmax}
Let $Q$ be an algebraic, connected, cocompact, solvable $\R$-split subgroup of $G$. Then, $Q$ is conjugate to $AN$.
\end{lemma}
\begin{proof}
By Lemma~\ref{AN}, we may assume, up to conjugation, that $Q$ contains $AN$. Moreover, since $Q$ is algebraic and solvable, it may be written as $Q = R_u(Q) \rtimes T$, where $T$ is a maximal torus of $Q$ \cite[Theorem 10.6 (4)]{Borel}. This torus is $\R$-split by assumption, and it is a maximal $\R$-split torus of $G$ because $A$ is. Up to conjugation, we may assume that $T = A$.

We also know that $N \subset R_u(Q)$, and $R_u(Q)$ decomposes as root spaces for the action of $A$. Consequently, the Lie algebra of $R_u(Q)$ contains a set of positive root spaces, which is $N$, and cannot contain a negative root space, because otherwise it would contain a subalgebra isomorphic to $\mathfrak{sl}(2,\R)$, which is not solvable. We deduce that $N = R_u(Q)$, proving the theorem.
\end{proof}

A final result needed later in our analysis concerns unimodularity.

\begin{proposition} \label{unimodular}
The torus $A$ contains an element acting by contraction on any subspace of $\mathfrak{n}$. In particular, this element acts by contraction on any $A$-invariant subgroup of $N$, and on any quotient of such a subgroup.
\end{proposition}
\begin{proof}
Let $a \in \mathfrak {a} \setminus \{\mathrm{0}\}$ such that $\alpha(a) < 0$ for all roots $\alpha \in \Sigma^+$. Such an $a$ exists by definition of positive roots. In addition, $a$ acts on $N$ by a diagonal matrix with negative coefficients. Passing to the exponential, $\exp(a)$ acts as a contraction on any subspace of $\mathfrak n$, thus it is a contraction.
\end{proof}

\subsection{GIB manifolds with symmetric part $N$} We consider a GIB manifold $\Gamma \backslash (\R^q \times N)$ and we assume that $(N, g_N)$ is a symmetric space of non-compact type. This situation has already been studied in \cite{FZ2}, and some classification results were established in the rank $1$ case. Here, we prove Theorem~\ref{P0UnipotentR}.

We recall that $\Gamma \subset \mathrm{Sim}(\R^q) \times \mathrm{Isom}(N,g_N)$, and the closure of the projection of $\Gamma$ onto $\mathrm{Isom}(N,g_N)$ is denoted by $\bar P$.

The group $I := \mathrm{Isom}(N,g_N)^0$ is commensurable to $\mathrm{Isom}(N, g_N)$. Replacing $\Gamma$ by a finite index subgroup if necessary, we may assume that $\bar P \subset I$. Applying the adjoint representation $\mathrm{Ad} : I \mapsto \mathrm{GL}(\mathfrak{I})$, where $\mathfrak{I}$ is the Lie algebra of $I$, we identify $I$ with its image. Indeed, $\mathrm{Ad}$ is into because $I$ is semisimple and has trivial center. In addition, $\mathrm{Ad}(I)$ is an algebraic subgroup of $\mathrm{GL}(\mathfrak g)$.

We also recall that ${\bar P}^0$ is the identity component of $\bar P$. It is an abelian Lie group, and we denote by $\mathfrak{p}$ its Lie algebra. We consider the normalizer of ${\bar P}^0$ in $G$:
\begin{equation}
\mathrm{Nor}_G({\bar P}^0) := \{ g \in G \ \vert \ g {\bar P}^0 g^{-1} = {\bar P}^0 \} = \{ g \in G \ \vert \ g \mathfrak{p} g^{-1} \subset \mathfrak{p} \}.
\end{equation}
Since stabilizing $\mathfrak{p}$ is tantamount to satisfying polynomial equations, we deduce that $\mathrm{Nor}_G({\bar P}^0)$ is Zariski closed in $G$. We also notice that $\bar P \subset \mathrm{Nor}_G({\bar P}^0)$. Since $\bar P$ acts cocompactly on $N$, one has that $\mathrm{Nor}_G({\bar P}^0)$ is cocompact in $G$.

We aim to identify the group ${\bar P}^0$ as the unipotent radical of $\mathrm{Nor}_G({\bar P}^0)$. We start by showing:

\begin{proposition} \label{normalize}
Let $H$ be an abelian connected subgroup of $G$. Let $Q$ be an algebraic cocompact subgroup of $G$ normalizing $H$. Then, $H$ is contained in the unipotent radical of $Q$.
\end{proposition}
\begin{proof}
Let $R_u(Q) \rtimes U$ be a Levi decomposition of $Q$ (see \cite{Mos56} and \cite[\textsection 11.22]{Borel}). Then, there exists an isogeny $U \simeq S \times T$ where $T$ is a torus and $S$ is semisimple, and one has the isogenies
\begin{equation} \label{decompositionL}
Q \simeq R_u(Q) \rtimes U \simeq R_u(Q) \rtimes (T \times S) \simeq (R_u(Q) \rtimes T) \rtimes S \simeq R(Q) \rtimes S,
\end{equation}
Where $R(Q)$ is the radical of $Q$. The Zariski closure $\overline{H}^{\mathrm{Zar}}$ of $H$ in $G$ is an abelian group because $H$ is abelian \cite[\textsection 2.4, Corollary 2]{Borel}. In particular, it is a product $R_0 \times T_0$ where $T_0$ is a torus and $R_0$ is the nilpotent radical of $\overline{H}^{\mathrm{Zar}}$ \cite[Theorem 10.6 (3)]{Borel}.

We claim that $R_0 \subset R_u(Q)$. We first remark that the projection of $\overline{H}^{\mathrm{Zar}}$ onto $S$ is a normal abelian subgroup of $S$ because $\overline{H}^{\mathrm{Zar}}$ is normal in $Q$, therefore this projection is contained in the center of $S$, which is finite. It follows that $\overline{H}^{\mathrm{Zar}}$ is included in $R_u(Q) \rtimes T$ because this group is closed and $\overline{H}^{\mathrm{Zar}}$ is connected. The claim follows from the fact that $R_u(Q)$ is exactly the set of unipotent elements of $R(Q)$.

The torus $T_0$ is included in a maximal torus of $R(Q)$, and this maximal torus is conjugate to $T$ by an element $u \in R_u(Q)$ (see \cite[Theorem 34.4]{Hum} and \cite[Theorem 19.2]{Borel}), i.e. $u T_0 u^{-1} \subset T$. Moreover, $\overline{H}^{\mathrm{Zar}}$ is normal in $Q$, so $u T_0 u^{-1}$ is a maximal torus of $\overline{H}^{\mathrm{Zar}}$. Since $\overline{H}^{\mathrm{Zar}}$ is abelian, it has a unique maximal torus, and $T_0= u T_0 u^{-1} \subset T$.

Let $t \in T_0$ and $u \in R_u(Q)$. The group $T_0$ is normal in $Q$ and normalizes $R_u(Q)$, thus the commutator
\[
[t,u] := t u t^{-1} u^{-1}
\]
belongs to both $T_0$ and $R_u(Q)$. Since $T_0 \cap R_u(Q) = \{ \mathrm{id} \}$, one has $[t,u] = \mathrm{id}$. Hence, $T_0$ centralizes $R_u(Q)$, thus it centralizes $Q$. By Corollary~\ref{Commute}, $T_0$ commutes with the whole group $G$, because $G$ has no compact factor. It follows that $T_0$ is finite because $G$ is semisimple, and ${\bar P}^0$ is contained in $R_0 \subset R_u(Q)$ because it is connected.
\end{proof}

Let $L$ be an algebraic subgroup of $G$ such that $\bar P \subset L \subset \mathrm{Nor}_G({\bar P}^0)$. Applying Proposition~\ref{normalize} to our particular setting, we obtain that ${\bar P}^0 \subset R_u(L)$.

Since ${\bar P}^0$ is connected, we may assume that all the groups are connected for the usual topology in what follows.

So far, we know that ${\bar P}^0 \subset R_u(L)$. We would like to prove equality. We observe that the decomposition of equation~\eqref{decompositionL} applies to $L$ and we write
\begin{equation}
\mathrm{Nor}_G({\bar P}^0) = R_u(L) \rtimes (S \times T).
\end{equation}
We introduce the group
\begin{equation}
R' := R_u(\mathrm{Nor}_G({\bar P}^0))/{\bar P}^0,
\end{equation}
and we consider
\begin{equation}
L' := L / {\bar P}^0 = R' \rtimes (S \times T).
\end{equation}
The group $\Gamma_1 := \bar P / {\bar P}^0$ is a discrete subgroup of $L'$. Moreover, it is cocompact, thus it is a lattice in $L'$.

Let $S_K$ be the maximal connected compact normal subgroup of $S$. By a theorem proved by Mostow \cite[Lemma 3.9]{Mos71} and presented by Auslander \cite[Theorem 2]{Aus}, the group
\begin{equation}
\Gamma' := \Gamma_1 \cap (R' \rtimes S_K)
\end{equation}
is a cocompact lattice in $R' \rtimes S_K$.

The group $\bar P$ acts by conjugation on $R' \rtimes S_K$ and preserves the lattice $\Gamma'$. Consequently, $\bar P$ acts unimodularly on $R' \rtimes S_K$. Since $\bar P$ is cocompact in $L$, the kernel of the modular form $\Delta : L \to \R^*_+$ (for the action of $L$ on $R_u(L) / {\bar P}^0$) is cocompact, and $\Delta$ factors  through this kernel to a group homomorphism from a compact group to $\R^*_+$. This latter homomorphism is thus constant and $L$ acts unimodularly on $R' \rtimes S_K$. In particular, $T$ acts unimodularly on $R' \rtimes S_K$, and thus on $R'$ because it commutes with $S_K$.

Let $S = K_S A_S N_S$ be an Iwasawa decomposition of $S$ and let $T_s$ be the maximal $\R$-split subtorus of $T$. We consider the group
\begin{equation}
\mathbf{P} := R_u(L) \rtimes ((N_S \rtimes A_S) \times T_s) = (R_u(L) \rtimes N_S) \rtimes (A_S \times T_s).
\end{equation}
It is a cocompact, solvable, $\R$-split algebraic subgroup of $G$, thus, if $G = K_G A_G N_G$ is the Iwasawa decomposition of $G$, applying Theorem~\ref{ANmax} we see that $\mathbf{P}$ is equal to $A_G N_G$ up to conjugation, with $A_G = A_S \times T$ and $N_G = R_u(L) \rtimes N_S$.

By the previous discussion, $T_s$ acts unimodularly on $R_u(L') = R_u(L) / {\bar P}^0$ and $A_S$ acts unimodularly on $R_u(L)$. Thus, the maximal $\R$-split torus $A_G$ acts unimodularly on $R_u(L)$. However, Proposition~\ref{unimodular} provides us with an element $a \in A_G$ acting by contraction on $R_u(L')$. Therefore, $R_u(L')$ must be trivial, i.e. $R_u(L) = {\bar P}^0$.

From the previous discussion, we deduce that ${\bar P}^0$ is the unipotent radical of any algebraic group between ${\bar P}^0$ and $\mathrm{Nor}_G({\bar P}^0)$. In particular, by \cite[Corollary 3.2]{BoTi71}, $\mathrm{Nor}_G({\bar P}^0)$ is a parabolic subgroup of $G$. This concludes the proof of Theorem~\ref{P0UnipotentR}.

\begin{remark}
While writing this article, we remarked  that Theorem~\ref{P0UnipotentR} might be partially seen as a consequence of \cite[Theorem 1]{Moo}. Our proof is, in addition, really close to that of the latter article. Calvin C. Moore claims that there is an explicit list of all parabolic subgroups of semisimple Lie groups with abelian unipotent radical. This list could provide a complete classification of the symmetric spaces occurring in our setting. Yet, the article cited in \cite{Moo} does not seem to contain the list.
\end{remark}

The monodromy subgroup of the GIB manifold is contained in the reductive part $S \times T$. It remains to identify the elements of $\Gamma_1 = \bar P / {\bar P}^0$ acting trivially by conjugation on ${\bar P}^0$. We first notice that $\Gamma_1$ injects into a Levi subgroup of $L$, namely $S \times T$. Let $t$ be an element of $T_s$ acting trivially on ${\bar P}^0$. In this case, $t$ commutes with $A_G N_G$, and, by Corollary~\ref{Commute}, it is contained in a compact factor of $G$.

\begin{lemma} \label{trivialElements}
The elements of $T$ acting trivially on ${\bar P}^0$ form a compact group.
\end{lemma}

We now prove Theorem~\ref{SymZar}.

\begin{lemma} \label{discreteproj}
Let $(n,H,V,b)$ be GIB arithmetic data. Then, the projection of the automorphism group $\mathrm{Aut}(n,H,V,b)$ on $\mathrm{GL}(V)$ is discrete.
\end{lemma}
\begin{proof}
Let $(h_n, v_n)_{n \in \N} \in (\mathrm{GL}(H) \times \mathrm{GL}(V))^\N$ be a sequence of elements lying in $\mathrm{Aut}(n,H,V,b)$, such that $(v_n)_\N$ converges to the identity. Then, $(\mathrm{det}(v_n))_{n \in \N}$ converges to $1$, and so does $(\mathrm{det}(h_n))_{n \in \N}$ because $\mathrm{Aut}(n,H,V,b) \subset \mathrm{GL}(n,\Z)$. Hence, $(\mathrm{det}(h_n))_{n \in \N}$ is bounded, and $(h_n)_\N$ is contained in a compact subset of $\mathrm{GL}(H)$ because it is a subsequence of $\mathrm{O}(b)$.

The sequence $(v_n)_\N$ is also contained in a compact neighbourhood of the identity, and we conclude that $(h_n, v_n)_\N$ is contained in a compact neighbourhood of $I_n$. But there are only finitely many elements of $\mathrm{Aut}(n,H,V,b)$ in such a neighbourhood. Consequently, $(h_n, v_n)_\N$ is constant after a certain rank, equal to $I_n$, implying the lemma.
\end{proof}

Let $\rho$ be the monodromy representation of the GIB manifold. Let $S = S_1 \times \ldots \times S_r$ be the decomposition of $S$ into simple factors. By Lemma~\ref{trivialElements}, the kernel of the action of $T \times S$ by conjugation on ${\bar P}^0$ is a product of some of these factors and a compact factor $T_K$ of $T$. Assume for simplicity that this kernel is $S_2 \times \ldots \times S_r \times T_K$. The group $\rho(\Gamma)$ might be viewed as a subgroup of $\mathrm{Sim} (\R^q) \times S_1 \times T$.

By Lemma~\ref{discreteproj}, the projection of $\rho(\Gamma)$ on $S_1 \times T$ is discrete. In addition, it coincides with the projection of $\Gamma$ on $S_1 \times T$, thus it is cocompact. The group $\overline{\rho(\Gamma)}^\mathrm{Zar}$ is a subgroup of $\mathrm{Sim} (\R^q) \times S_1 \times T$. Applying Lemma~\ref{discreteproj} again, the integer points of $\overline{\rho(\Gamma)}^\mathrm{Zar}$ projects as a discrete subgroup of $S_1 \times T$, containing the projection of $\rho(\Gamma)$. Since $\Gamma$ is isomorphic to its projection on $\mathrm{Isom}(N, g_N)$ (see the discussion preceding Proposition 2.1 in \cite{FZ2}), this shows that $\rho(\Gamma)$ has finite index in $\overline{\rho(\Gamma)}^\mathrm{Zar}(\Z)$. This last group is cocompact in $\overline{\rho(\Gamma)}^\mathrm{Zar}$, and so is $\rho(\Gamma)$.

It remains to prove the injectivity of the monodromy when $(N, g_N)$ is irreducible. In this case, the isometry group $\mathrm{Isom}(N,g_N)$ is simple. We write $S_1 \times \ldots \times S_r$ for the decomposition of $S$ into simple factors. If a noncompact factor $S_i$ of $S$, $1 \le i \le r$, acts trivially on ${\bar P}^0$, then the root system of $G$ splits, and $G$ is not simple. Consequently, the subgroup of $S$ acting trivially on ${\bar P}^0$ is a compact factor. By Corollary~\ref{Commute} it is a compact factor of $G$, but $G$ has no such factor since it is trivial. Therefore, the monodromy representation of the GIB manifold is injective.

This proves Theorem~\ref{SymZar}.

\section{Explicit examples of matrix groups} \label{SecExamples}

In this section, we give some explicit constructions GIB manifold where the factor $N$ is a symmetric space. We construct them using matrix groups. We also discuss the classification of GIB arithmetic data from the point of view of matrices.

\subsection{Examples of suitable symmetric spaces} According to Theorem~\ref{P0UnipotentR}, when the factor $N$ of the universal cover of a GIB manifold is a symmetric space of non-compact type, the normalizer of ${\bar P}^0$ is a parabolic subgroup of $\mathrm{Isom}(N)^0$, and ${\bar P}^0$ is its unipotent radical. Therefore, we should search for semisimple Lie groups with finite center admitting such parabolic subgroups. This section is devoted to the description of a few examples.

\begin{Ex} \label{CounterEx}
Let $m > 0$ be an integer. The group $\mathrm{SL}(m+1, \R)$ contains the parabolic subgroup
\begin{equation}
\mathbf P := \R^m \rtimes (\R^* \times \mathrm{SL}(m, \R)).
\end{equation}
If $(e_1, \ldots, e_{m+1})$ is a basis of $\R^{m+1}$, $\mathbf{P}$ might be viewed as the stabilizer of the flag $(\mathrm{Span}(e_1, \ldots, e_m), \mathrm{Span}(e_1, \ldots, e_{m+1}))$. It is explicitly given by the injections
\begin{equation*}
\begin{aligned} [c]
\R^m &\longrightarrow \mathrm{SL}(m+1, \R) \\
v &\longmapsto \left( \begin{matrix} I_m & v \\ 0 & 1 \end{matrix} \right)
\end{aligned}
,\quad
\begin{aligned} [c]
\R^* &\longrightarrow \mathrm{SL}(m+1, \R) \\
\lambda &\longmapsto \left( \begin{matrix} \vert \lambda \vert^{-1/m} I_m & 0 \\ 0 & \lambda \end{matrix} \right)
\end{aligned}
,\quad
\begin{aligned}
\mathrm{SL}(m, \R) & \longrightarrow \mathrm{SL}(m+1, \R) \\
A &\longmapsto \left( \begin{matrix} A & 0 \\ 0 & 1 \end{matrix} \right)
\end{aligned}.
\end{equation*}
In this situation, the symmetric space $N$ would be $\mathrm{SL}(m+1, \R)/\mathrm{SO}(m+1, \R)$, and ${\bar P}^0$ would be $\R^m$. However, we show that this does not give rise to a GIB manifold unless $m>1$. We prove this claim by contradiction, assuming the existence of a GIB manifold $M$ with universal cover $\R^q \times N$ ($q>1$) and ${\bar P}^0 = \R^m$.

The parabolic group $\mathbf P$ is the only algebraic, connected, cocompact subgroup of $\mathbf P$ having $\R^m$ as its unipotent radical. Therefore, it is equal to the Zariski closure of $\bar P$ by Theorem~\ref{P0UnipotentR}. We also notice that no element of $\R^* \times \mathrm{SL}(m, \R)$ centralizes $\R^m$, so the monodromy map is injective. In particular, the Zariski closure of $\Gamma := \bar P/{\bar P}^0$ is isomorphic to $\R^* \times \mathrm{SL}(m, \R)$, and it coincides with the Zariski closure of the monodromy group of $M$. Hence, Theorem~\ref{structureThm0} implies the existence of a cocompact arithmetic subgroup $D$ of $\mathrm{SL}(m, \R)$ acting on $\R^q \oplus \R^m$, preserving the decomposition, and whose restriction to $\R^q$ is relatively compact. By \cite[Proposition 5.5.12]{Mor}, this arithmetic group is obtained by restriction of scalars (see Proposition~\ref{restrictionscalars}).

Since the only faithful representation of $\mathrm{SL}(m, \R)$ on $\R^m$ is the canonical one, we deduce that the Zariski closure of $D$ is isogenous to a product
\[
\prod\limits_{\sigma \in S^\infty} (\mathrm{SL}(m,\R))^\sigma,
\]
where $S^\infty$ is the set of archimedean places of some algebraic number field $k$. But the terms of this product are either $\mathrm{SL}(m, \R)$ or $\mathrm{SL}(m, \C)$. This is possible only if $k = \Q$, implying $S \simeq \mathrm{SL}(m,\Z)$. This latter arithmetic subgroup is not cocompact if $m >1$, which is a contradiction.
\end{Ex}

\begin{Ex} \label{SOpq}
Let $p$,$r$ be two integers. The group $\mathrm{SO}(p+1,r+1)$ contains the parabolic subgroup
\begin{equation}
\mathbf P := \R^{p,r} \rtimes (\R^* \times \mathrm{SO}(p,r)).
\end{equation}
This is the stabilizer of a light-like line in $\R^{p+1,r+1}$, and we have the injections
\begin{equation*}
\begin{aligned}
\R^{p,r} &\longrightarrow \mathrm{SO}(p+1,r+1) \\
v & \longmapsto \left( \begin{matrix} 1 & - v^T \eta & - \frac{1}{2} \Vert v \Vert_{p,r} \\ 0 & I_{p+r} & v \\ 0 & 0 & 1 \end{matrix} \right)
\end{aligned},
\quad
\begin{aligned}
\R^* &\longrightarrow \mathrm{SO}(p+1,r+1) \\
\lambda & \longmapsto \left( \begin{matrix} \lambda & 0 & 0 \\ 0 & I_{p+r} & 0 \\ 0 & 0 & \lambda^{-1} \end{matrix} \right)
\end{aligned},
\end{equation*}
\begin{equation*}
\begin{aligned}
\mathrm{SO}(p,r) &\longrightarrow \mathrm{SO}(p+1,r+1) \\
A & \longmapsto \left( \begin{matrix} 1 & 0 & 0 \\ 0 & A & 0 \\ 0 & 0 & 1 \end{matrix} \right)
\end{aligned},
\end{equation*}
where we endowed $\R^{p+1,r+1}$ with the quadratic form given by the matrix
\begin{equation}
\eta := \left( \begin{matrix} 0 & 0 & 0 &1 \\ 0 & I_p & 0 & 0 \\ 0 & 0 & - I_r & 0 \\ 1 & 0 & 0 & 0 \end{matrix} \right).
\end{equation}

Here, the symmetric space is $N := \mathrm{SO}(p+1,r+1)/\mathrm{S}(\mathrm{O}(p+1) \times \mathrm{O}(r+1))$, and we can define a GIB manifold with these data. Indeed, let $b$ be the quadratic form on $\R^{p+r}$ defined by
\begin{equation}
b := x_1^2 + \ldots + x_p^2 - \sqrt{2} (x_{p+1}^2 + \ldots + x_{p+r}^2).
\end{equation}
Taking the algebraic field $k := \Q[\sqrt{2}]$ and denoting by $\sigma$ its non-trivial embeddings, we obtain an arithmetic subgroup of $\mathrm{SO}(p,r)$ by considering
\begin{equation}
\Gamma' := \{ (\sigma(A), A) \ \vert \ A \in \mathrm{SO}(b, \Z[\sqrt{2}]) \}.
\end{equation}
The group $\Gamma'$ acts naturally on $\R^{p+r} \oplus \R^{p,r}$ and its restriction to the first factor of this decomposition is contained in $\mathrm{SO}(\sigma(b),\R) \simeq \mathrm{SO}(p+r, \R)$. Moreover, $\Gamma'$ is conjugate to a subgroup of $\mathrm{GL}(2(p+r), \Z)$, so it preserves a lattice $\Gamma_0$ in $\R^{p+r} \oplus \R^{p,r}$. Pick $\lambda := \frac{2 + \sqrt{2}}{2}$ in the subgroup $\R^*$ of $\mathbf P$, and remark that it acts by multiplication by $\lambda$ on the unipotent radical $\R^{p,r}$ of $\mathbf P$. Let $A_0$ be the matrix acting as $(x,y) \mapsto \lambda^{-1}(x, \lambda y)$ on $\R^{p+r} \oplus \R^{p+r}$. Finally, we define the group
\begin{equation}
\Gamma := \Gamma_0 \rtimes (\langle A_0 \rangle \times \Gamma'),
\end{equation}
acting on $\R^{p+r} \times N$ in the following way:
\begin{equation}
\begin{aligned}
(\Gamma_0 \times \langle A_0 \rangle \times \Gamma') \times (\R^{p+r} \times N) & \longrightarrow \R^{p+r} \times N \\
((v,w), A_0^m, (\sigma(A), A)), (x, y) & \longmapsto (\lambda^{-m} \sigma(A) x + v, w \lambda^m A \cdot y)
\end{aligned},
\end{equation}
where we see $\Gamma_0$ as a subset of $\R^{p+r} \oplus \R^{p,r}$.

By definition, $\Gamma$ is a discrete subgroup of $\mathrm{Sim}(\R^{p+r}) \times \mathrm{Isom}(N)$ acting properly and cocompactly on $\R^{p+r} \times N$. However, it might not act freely, but if the action is not free, it means that its restriction to $N$ is not free, and equivalently there is a matrix in $\Gamma'$ whose restriction to $N$ is a torsion element. Selberg's lemma allows us to remove such elements from $\Gamma'$.

In conclusion, $\Gamma \backslash (\R^{p+r} \times N)$ is a GIB manifold.
\end{Ex}

\subsection{An explicit matrix description} In this section, we provide a way to construct GIB arithmetic data explicitly. We also explain how to find the polynomials defining the Zariski closure of the automorphism group. These results follow directly from the definition, but we believe it is important to write them down.

\subsubsection{Classification of GIB data} Let $(n, H, V, b)$ be GIB arithmetic data as presented in Definition~\ref{GIBdata}. Let $A$ be a matrix in $\mathrm{Aut}(n,H,V,b)$ whose restriction to $H$ is not a $b$-isometry.

It was proved in \cite[Section 5.2]{Fla25} that there exists a basis of the lattice $\Z^n$ in which $A$ is diagonal by blocks, and the characteristic polynomials of the blocks are irreducible over $\Q$. We write $A = \mathrm{Diag} (A_1, \ldots, A_m)$, where the $A_i$ are the blocks in the adapted basis.

Since $A$ restricts to a similarity on $H$, and $H$ is an irrational subspace of $\R^n$, there exists a positive real number $\lambda \neq 1$ such that each block has an eigenvalue of modulus $\lambda$ (otherwise $H$ would not be irrational) and $H$ is a subspace of the direct sum of the eigenspaces of $A$ with eigenvalue of modulus $\lambda$.

Since the blocks $A_i$ have irreducible characteristic polynomials, either they have the exact same eigenvalues or they have no common eigenvalue. We may assume that the blocks are sorted in the following way: there are indices $1 = i_1 < \ldots < i_{\ell+1} = m+1$ such that $A_i$ and $A_j$ have the same characteristic polynomial if and only if there exists $1 \le r \le \ell$ with $i_{r}\le i,j < i_{r+1}$. We now see the matrix $A$ as a matrix by blocks $\mathrm{Diag}(A'_1, \ldots, A'_\ell)$ where $A'_r = \mathrm{Diag}(A_{i_r}, \ldots, A_{i_{r+1}-1})$. Consequently, we may assume that all blocks $A_i$ have the same irreducible characteristic polynomial in what follows.

We recall that $A$ restricts to $H$ as a similarity of ratio $\lambda \neq 1$. Let $a_1, \ldots, a_s$ be the real eigenvalues of $A_1$ and let $a_{s+1}, \overline{a_{s+1}}, \ldots, a_{s+t}, \overline{a_{s+t}}$ be the complex ones. Identifying a complex number $a e^{i \theta}$ ($a, \theta \in \R_+ \times [0, 2 \pi]$), $\theta \neq 0$, to the real matrix
\begin{equation}
a \left( \begin{matrix} \cos(\theta) & -\sin(\theta) \\ \sin(\theta) & \cos(\theta) \end{matrix} \right),
\end{equation}
all the blocks $A_i$ are similar to $A_0 := \mathrm{Diag} (a_1, \ldots, a_{s+t})$ over $\C$. Hence, $A$ might be seen as the linear transformation $A_0 \otimes I_m$ acting on $\R^{s+2t} \otimes \R^m$. Let
\begin{equation}
\Lambda := \{ a_j \ \vert \ 1 \le j \le s+t, \ \vert a_j \vert = \lambda \},
\end{equation}
and for any $a \in \Lambda$, let $E_a$ be the eigenspace of $A_0$ for the eigenvalue $a$ (if $a$ is complex, this is a complex line, identified with $\R^2$). Then, there exist subspaces $F_a$ (possibly trivial), $a \in \Lambda$, of $\R^m$ such that
\begin{equation} \label{defH}
H = \bigoplus\limits_{a \in \Lambda} E_a \otimes F_a. 
\end{equation}
The subspace $H$ is irrational in $\R^{s+2t} \otimes \R^m$ if and only if $\bigoplus\limits_{a \in \Lambda} F_a + \Z^m$ is dense in $\R^m$.
For any $a \in \Lambda$, there exists a supplementary space $F_a'$ to $F_a$ in $\R^m$ such that
\begin{equation} \label{defV}
V = \bigoplus\limits_{a \in \Lambda} E_a \otimes F_a' \oplus \bigoplus\limits_{a \notin \Lambda} E_a.
\end{equation}

Conversely, let $A \in \mathrm{GL}(n,\Z)$, diagonal by blocks with all blocks having irreducible characteristic polynomials. Assume moreover that there exists $\lambda > 1$ such that each block has at least one eigenvalue of modulus $\lambda$. Now, we group the blocks of $A$ as before, and for simplicity we assume that all the blocks have the same irreducible characteristic polynomial. Then, we define $H$ and $V$ as in Equations~\eqref{defH} and \eqref{defV}, where the subspaces $F_a$ are chosen so that $\bigoplus\limits_{a \in \Lambda} F_a + \Z^m$ is dense in $\R^m$. Then, $H$ is irrational in $\R^{s+2t} \otimes \R^m$, $A$ preserves the decomposition $H \oplus V$, and the restriction of $A$ to $H$ is diagonalizable, with all its eigenvalues having modulus $\lambda$. Consequently, one may find an inner product $b$ on $H$ such that $A$ acts as a similarity of ratio $\lambda$ on $(H,b)$. We deduce that, for a suitable integer $n$, $(n,H,V,b)$ are GIB arithmetic data.

All the GIB arithmetic data are obtained by the construction above, so we have a complete description of these structures.

\subsubsection{Zariski closure of the automorphism group} Let $(n,H,V,b)$ be GIB arithmetic data. As we saw before, the Zariski closure of $\mathrm{Aut}(n,H,V,b)$ is defined over $\Q$, and so is its commutator subgroup. We explain here how one can find explicitly the polynomial equations that a matrix should satisfy in order to belong to $\mathrm{Aut}(n,H,V,b)$.

Let $A \in \mathrm{Aut}(n,H,V,b)$. Then, $A$ must preserve the decomposition $H \oplus V$. Denoting by $p$ the projector on $H$ parallel to $V$, this condition is equivalent to
\begin{equation} \label{cond1}
p A = A p.
\end{equation}
In addition, the restriction of $A$ to $H$ must preserve the inner product $b$. Equivalently,
\begin{equation} \label{cond2}
b (A \vert_H \cdot, A \vert_H \cdot) = \mathrm{det}(A \vert_H)^2 b.
\end{equation}
Consequently, a matrix $A \in \mathrm{GL}(n,\Z)$ belongs to $\mathrm{Aut}(n,H,V,b)$ if and only if it satisfies equations~\eqref{cond1} and \eqref{cond2}.

Let $E$ be the smallest $\Q$-vector subspace of $\mathrm{Mat}(n, \Q)$ such that $E \otimes_\Q \R$ contains $p$. Let $(A_1, \ldots, A_m)$ be a $\Q$-basis of $E$. Then, there are real numbers $\alpha_1, \ldots, \alpha_m$ such that
\[
p = \sum\limits_{i=1}^m \alpha_i A_i.
\]
The minimality condition on $E$ implies that the $\alpha_i$ are independent over $\Q$. Indeed, if
\[
\alpha_1 = \lambda_2 \alpha_2 + \ldots + \lambda_m \alpha_m
\]
with $\lambda_i \in \Q$, then
\[
p = \sum\limits_{i=2}^m \alpha_i (\lambda_i A_1 + A_i),
\]
and $\lambda_i A_1 + A_i \in \mathrm{Mat}(n, \Q)$. Therefore, for $A \in \mathrm{GL}(n,\Z)$, the condition $p A = A p$ is equivalent to
\begin{equation}
\sum\limits_{i=1}^m \alpha_i (A_i A - A A_i) = 0,
\end{equation}
which, in turn, amounts to
\begin{align} \label{finalcond1}
[A_i, A] = 0 \text{ for all $1 \le i \le m$}.
\end{align}

The second condition can be handled in the same way for matrices restricting to isometries of $(H,b)$. First, we view $b$ as a quadratic form on $H \oplus V$ vanishing on $V$. Hence, $b$ is represented by a symmetric matrix $q$ and condition~\eqref{cond2} is equivalent to $A \vert_H^T q A = \mathrm{det}(A \vert_H)^2 q$. Let $F$ be the minimal $\Q$-vector subspace of $\mathrm{Mat}(n,\Q)$ such that $F \otimes_\Q \R$ contains $q$. We fix a $\Q$-basis $(Q_1, \ldots, Q_r)$ of $F$, and proceeding exactly as before, we show that the condition $A^T q A = q$ is equivalent to
\begin{align} \label{finalcond2}
A^T Q_i A = Q_i \text{ for all $1 \le i \le r$}.
\end{align}

We conclude that $A \in \mathrm{GL}(n,\Z)$ is in $\mathrm{Aut}(n,H,V,b)$ and restricts to an isometry of $(H,b)$ if and only if $A$ satisfies both conditions~\eqref{finalcond1} and \eqref{finalcond2}. This gives us the precise rational polynomial equations that a matrix should satisfy in order to be in the automorphism group and to restrict to an isometry on $H$.

\end{document}